\documentclass[12pt]{article}
\RequirePackage[OT1]{fontenc}
\RequirePackage[numbers]{natbib}

\usepackage{graphicx,amsmath,amsfonts,amssymb,amsthm,psfrag, framed,thmtools,ulem}
\usepackage[dvipsnames]{xcolor}
\usepackage{fullpage}
\declaretheoremstyle[
spaceabove=6pt, spacebelow=6pt,
headfont=\normalfont\bfseries,
notefont=\mdseries, notebraces={(}{)},
headpunct=.\,— ,
bodyfont=\normalfont,
numbered=no
]{solu}

\usepackage{comment}
\newcommand{\ds}{\displaystyle}

\newcommand{\dint}{\mathrm{d}}

\newcommand{\R}{\mathbb{R}}

\includecomment{comment} %show comments
\newtheorem{thm}{Theorem}[section]

\newtheorem{lemma}[thm]{Lemma}
\newtheorem{proposition}[thm]{Proposition}

\newtheorem{definition}[thm]{Definition}

\newtheorem{rem}[thm]{Remark}

\numberwithin{equation}{section}
\newtheorem{assu}[thm]{Assumption}

\title{Strong approximation of some particular \\ one-dimensional diffusions}
\date{}
\begin{document}
\author{Madalina Deaconu$^1$ and Samuel Herrmann$^2$
\\[5pt]
\small {$^1$Universit\'e de Lorraine, CNRS, Inria, IECL, F-54000 Nancy, France,}\\
\small{Madalina.Deaconu@inria.fr}\\[5pt]
\small{$^2$Institut de Math{\'e}matiques de Bourgogne (IMB) - UMR 5584, CNRS,}\\
\small{Universit{\'e} de Bourgogne Franche-Comt\'e, F-21000 Dijon, France} \\
\small{Samuel.Herrmann@u-bourgogne.fr}
}

\maketitle

%\tableofcontents

\begin{abstract}

We develop a new technique for the path approximation of one-dimensional stochastic processes. Our results apply to the  Brownian motion and to some families of 
stochastic differential equations whose distributions could be represented as a function of a time-changed Brownian motion 
(usually known as $L$ and $G$-classes). We are interested here in the $\varepsilon$-strong approximation. We propose an explicit and quite easy to implement procedure that constructs jointly, the sequences of exit times and corresponding exit positions of some well chosen domains. We prove in our main results the convergence of our scheme and how to control the number of steps which depends in fact on the covering of a fixed time interval by intervals of random sizes. The underlying idea of our analysis is to combine results on Brownian exit times from time-depending domains (one-dimensional heat balls) and classical renewal theory. Numerical examples and issues are also developed in order to complete the theoretical results. 
\end{abstract}

{
\noindent \textbf{Key words:} Strong approximation, path simulation, Brownian motion, linear diffusion.
\par\medskip

\noindent \textbf{2010 AMS subject classifications:} primary 
65C05;   	%Monte Carlo methods
secondary 
60J60,       %Diffusion processes
60J65,       %Brownian motion
60G17.      %Sample path properties	
}%

\section*{An introduction to strong approximation} 
Let $(X_t)_{ t\ge 0}$ be a stochastic process defined on the filtered probability space $(\Omega,{\cal{F}}, ({\cal{F}}_t)_t, \mathbb{P})$ and  $T$ be a fixed positive time. The aim of this study is to develop a new path approximation of $(X_t)_{\ 0\le t\le T}$ where $X_t$ stands either for the  one-dimensional Brownian motion starting in $x$ or for a class of one-dimensional diffusions with non-homogeneous coefficients.   \\
The usual and classical approximation procedure of any diffusion process consists in  constructing numerical schemes like the Euler scheme: the time interval is split into sub-intervals  $0<\frac{T}{n}<\ldots<\frac{n-1}{n}\, T <T$. For each of these time slots, the value of the process is given or approximated. The convergence result of the proposed approximation is then based on classical stochastic convergence theorems: we obtain usually some $L^p$-convergence between the path built by the scheme and the real path of the process. The approximation error is not a.s. bounded  by a constant. \\
In this study, we focus our attention on a different approach: for any $\varepsilon>0$, we construct a suitable sequence of increasing random times $(s_n^\varepsilon)_{n\ge 0}$ with $s^\varepsilon_0=0$, $\lim_{n\to \infty}s_n^\varepsilon=+\infty$  on the space $(\Omega,{\cal{F}},  \mathbb{P})$, and random points $x_0^\varepsilon,x_1^\varepsilon,\ldots,x_n^\varepsilon,\ldots$ in such a way that the random variable $x_n^\varepsilon$ is ${\cal{F}}_{s_n^\varepsilon}$ adapted for any $n\geq 0$ and
\begin{equation}
\label{eq:bound}
\sup_{{t}\in[0,T]}| X_t-x^\varepsilon_t |\le \varepsilon\quad \mbox{a.s.}
\end{equation}
where $x^\varepsilon_t=\sum_{n\ge 0}x_n^\varepsilon1_{\{  s_n^\varepsilon \le t<s^\varepsilon_{n+1}\}}$. The procedure is quite simple to describe, the sequence $(s^\varepsilon_n,x^\varepsilon_n)$ is associated to exit times and exit locations of well-chosen time-space domains
 for the process $(t,X_t)$. 

\noindent  We sketch the main steps of the method here. For this, let us consider a continuous function $\phi_\varepsilon(t)$ which satisfies: there exists $r_\varepsilon>0$ s.t. \\[5pt]
\centerline{${\rm Supp}(\phi_\varepsilon)=[0,r_\varepsilon]$ and $0<\phi_\varepsilon(t)\le \varepsilon$ for any $t\in \mathring{{\rm Supp}}(\phi_\varepsilon)$.}

\vspace*{0.2cm}
\noindent We start with $(s_0^\varepsilon,x^\varepsilon_0)=(0,x)$ where $x$ is the initial point of the path $(X_t)$. Then we define, for any $n\geq 0$
\[
s^\varepsilon_{n+1}:=\inf\{ t\ge s^\varepsilon_n:\  | X_t-X_{s^\varepsilon_n}| \ge \phi_\varepsilon(t-s^\varepsilon_n)\}
\]
and $x^\varepsilon_{n+1}:=X_{s^\varepsilon_{n+1}}$. In other words, $s^\varepsilon_{n+1}$ is related to the first exit time of the stochastic process $(t,X_{s^\varepsilon_{n}+t}-x^\varepsilon_{n})_{t\ge 0}$  from the time-space domain $\{(t,x)\in \mathbb{R}_+\times \mathbb{R} :\ |x|\le \phi_\varepsilon(t)\}$, called $\phi_\varepsilon$-domain in the sequel.

We observe that:
\begin{itemize}
\item[-] as the function $\phi_\varepsilon$ is bounded, the bound \eqref{eq:bound} is satisfied 
\item[-] as the $\phi_\varepsilon$ has a compact support, the sequence $(s^\varepsilon_n)$ satisfies $s^\varepsilon_{n+1}-s^\varepsilon_n\le r_\varepsilon$, for any $n\geq0$.
\end{itemize}
For such an approximation of the paths, the challenge consists in the choice of an appropriate function $\phi_\varepsilon$ defining the $\phi_\varepsilon$-domain in such a way that the simulation of both the exit time and the exit location is easy to construct and implement. Moreover the analysis of the random scheme is based on a precise description of the number of random intervals $[s^\varepsilon_n,s^\varepsilon_{n+1}[$ required in order to cover $[0,T]$. Such an analysis is developed in the next section.

Our main motivation is to develop a new approach that gives the $\varepsilon$-strong approximation for a large class of multidimensional SDEs. In this paper the main tools and results of this topic are developed for some particular SDEs in one dimension. We intend to pursue this research for more general situations starting with the multidimensional Brownian motion and Bessel processes.

The study of the strong behaviour of an approximation scheme, and in particular the characterisation by some bounds depending on $\varepsilon$ of $\sup_{t\in[0,T]}\Vert  X_t-x^\varepsilon_t \Vert$ where $x^\varepsilon_t $ stands for an approximation scheme,  was considered recently by some other authors. In Blanchet, Chen and Dong \cite{blanchet-chen-dong-2017} the authors study the approximation of multidimensional SDEs by considering transformations of the underlying Brownian motion (the so-called It\^o-Lyons map) and follow a rough path theory approach. In this paper the authors refer to the class of procedures which achieve the construction of such an approximation as Tolerance-Enforced Simulation (TES) or $\varepsilon$-strong simulation methods.  In Chen and Huang \cite{chen-huang-2013} a similar question is considered but the result is obtained only for SDEs in one dimension and the effective construction of an approximation scheme is not obvious. This last procedure was extended by Beskos, Peluchetti and Roberts \cite{beskos-peluchetti-roberts-2012} were an iterative sampling method, which delivers upper and lower bounding processes for the Brownian path, is given.  Let us finally mention the recent manuscript \cite{chen2019epsilonstrong} which highlights  an adaptation of such an approach to the fractional Brownian motion framework.

In a more general context, Hefter, Herzwurm and M\"uller-Gronbach \cite{hefter-herzwurm-mullergronbach-2019} give lower error bounds for the pathwise approximation of scalar SDEs, the results are based on the observations of the driving Brownian motion. Previously the notion of strong convergence was studied also intensively for particular processes like the CIR process. Strong convergence without rate was obtained by Alfonsi \cite{alfonsi-2005} or Hutzenthaler and Jentzen \cite{hutzenthaler-jentzen-2015}. Optimal lower and upper bounds were also given. For stochastic differential equations with Lipschitz coefficients M\"uller-Gronbach \cite{mullergronbach-2002} and
Hofmann, M\"uller-Gronbach and Ritter \cite{hofmann-al-2002} obtained lower error bounds. 
 
All these results give a new and interesting highlight in this topic of pathwise and $\varepsilon$-strong approximation,
%\sout{but are not quite effective for a numerical purpose due to the limitations of current techniques.} 
and prove how such an approach becomes an essential tool in the numerical approximation of SDEs. The procedure that we point out in this paper totally belongs to this promising field: we give an explicit and constructive procedure for the approximation of some particular SDEs. 
The corresponding numerical scheme is easy to implement and belongs both to the family of free knot spline approximations of scalar diffusion paths and to the family of $\varepsilon$-strong approximations. The complexity of the scheme is therefore directly linked to the number of knots required in order to describe precisely the stochastic paths on a given time-interval: Creutzig, M\"{u}ller-Gronbach and Ritter \cite{creutzig2007free} pointed out the smallest possible average sup-norm error depending on the average number of free knots. The important feature of the new approach of the current work is to emphasize a efficient scheme and its complexity, especially worthy for many applications. The method is essentially based on explicit distributions of the exit time 
for time-space domains, closely related to the behaviour of the underlying process. By performing a rigorous analysis we identify sharp estimates of the number of free knots. The analysis of the number of free knots is central in our approach. 
%\sout{We believe that the construction  we give and the results we prove in this paper provide an important step in the  development of the $\varepsilon$-strong convergence for general SDEs. The main advantage of our approach is that we give an explicit  and constructive procedure of the scheme and this conducts to an easy to implement algorithm. Furthermore the construction is based on explicit distributions of the exit time  for time-space domains, closely related to the behaviour of the underlying process. This construction is thus  deeply guided with the dynamic of the process.}

For practical purposes the approximation scheme is the object of interest and in order to characterize and control its behaviour we are looking for sequences which have the same distribution. We need thus to introduce the following definition:
\begin{definition}  Let $(\Omega, {\cal{F}},\mathbb{P})$ be a probability space and $(X_t)$ be a stochastic process on this space. The random process $(y_t^\varepsilon)$ is an $\varepsilon$-strong approximation of the stochastic process $(X_t)$ if there exists a stochastic process $(x_t^\varepsilon)$ on $(\Omega, {\cal{F}},\mathbb{P})$  satisfying \eqref{eq:bound} such that $(y_t^\varepsilon)$ and $(x_t^\varepsilon)$ are identically distributed.
\end{definition}

The material is organized as follows. In Section \ref{sec:nbr}, we focus our attention on the number of space-time domains used for building the approximated path on a given fixed time interval $[0,T]$. This number is denoted by $N_T^\varepsilon$. The main specific feature related to our approach is the randomness associated to the time splitting. A sharp description of the random number of time steps $N_T^\varepsilon$ permits to emphasize the efficiency of the $\varepsilon$-strong simulation. The first section points some information in a quite general framework, that is $\varepsilon^2\mathbb{E}[N_T^\varepsilon]$ is upper bounded in the $\varepsilon$ small limit, while the forthcoming sections permit to go into details for specific diffusion processes.  Section \ref{sec:brow} introduces the particular Brownian case and families of one-dimensional diffusions ($L$-class and $G$-class of diffusion in particular) are further explored in Section \ref{sec:diff}. In each case, an algorithm based on a specific $\phi_\varepsilon$-domain (heat ball) is presented (Theorem \ref{thm:BM} and Theorem \ref{thm:General}) and the efficiency of the approximation is investigated (Proposition \ref{prop:ren} and Theorem \ref{thm:efficient:diff}). We obtain the convergence towards an explicit limit for the average expression $\varepsilon^2\mathbb{E}[N_T^\varepsilon]$ as $\varepsilon$ tends to $0$. In the particular diffusion case, there exists a constant $\mu>0$ such that
\begin{equation}
\label{eq:thm:diff:efficient-intro}
\lim_{\varepsilon\to 0}\varepsilon^2\,\mathbb{E}[N^\varepsilon_T]=\mu\,\mathbb{E}\left[ \ds\int_0^{\rho(T)}  \frac{1}{\eta^2(x+B_s)}\dint s\right],\quad \forall (T,x)\in\R_+\times\R
\end{equation}
where $\eta$ and $\rho$ are both functions related to the approximation procedure and $(B_t)_{t\ge 0}$ stands for a standard one-dimensional Brownian motion.

Finally numerical examples permit to illustrate the convergence result of the algorithm in the last section.

\section{Number of random intervals needed for covering the time interval $[0,T]$}
\label{sec:nbr}
The sharpness of the approximation is deeply related to the number of random intervals $[s^\varepsilon_n,s^\varepsilon_{n+1}[$ used to cover $[0,T]$. If $(X_t)$ is a homogeneous Markovian process, then we observe that $U_{n+1}^\varepsilon=s_{n+1}^\varepsilon-s_n^\varepsilon$, for $n\ge 0$, is a sequence of i.i.d. a.s. bounded variables. Obviously we have:
\begin{equation*}
s_n^\varepsilon =\ds\sum_{i=1}^n U_i^\varepsilon.
\end{equation*}
 The number of variates corresponds to
\[
N_T^\varepsilon:=\inf\{ n\ge 1:\ s^\varepsilon_n\ge T \}.
\]
We  can control, for any $j\in\mathbb{N}$ and  $\lambda>0$:
\begin{align}\label{eq:markov}
\mathbb{P}(N_T^\varepsilon>j)=\mathbb{P}(s_j^\varepsilon<T)=\mathbb{P}(e^{-\lambda s_j^\varepsilon}>e^{-\lambda T})\le e^{\lambda T}\mathbb{E}[e^{-\lambda s_j^\varepsilon}]=e^{\lambda T}\mathbb{E}[e^{-\lambda U_1^\varepsilon}]^j.
\end{align}
This calculus proves that the upper-bound essentially depends on the Laplace transform of $U^\varepsilon_1$.

%\begin{nota} 
\label{constantes}
Before stating a first result let us give an important convention. All along the text we need to control (upper or lower bounds) several quantities. In order to do this we use $C$ and $\kappa$ to design positive constants, whose value may change from one line to the other. 
When the constants depend on parameters of prime interest, we use, for example, the notation $C_{T,\alpha }$  to suggest that the constant $C$ depends in some way on $T$ and $\alpha$, where $T$ and $\alpha$ denote here some parameters.

%\end{nota}

\begin{proposition}\label{lem:lem2} For $\varepsilon >0$, let us assume that $U^\varepsilon_1 \stackrel{(d)}{=} \varepsilon^2 U$ where $U$ is a positive random variable which does not depend on the parameter $\varepsilon$.\\
%, (here and all along the paper $\stackrel{(d)}{=} $ stands for equality in distribution).\\
1. If there exist two constants $C>0$ and $\kappa>0$ such that $\mathbb{E}[e^{-\lambda U}]\le \frac{C}{\lambda^\kappa}$ for all $\lambda>0$, then 
\begin{equation}\label{eq:upper-prob}
\mathbb{P}(N_T^\varepsilon>j)\le \left( \frac{eTC^{1/\kappa}}{j\kappa\varepsilon^2} \right)^{j\kappa},\quad \forall j\in\mathbb{N}.
\end{equation}
2. If $\mathbb{E}[U^2]<\infty$, then for any $\delta>1$, there exists $\varepsilon_0>0$ such that 
\[
\mathbb{E}[N_T^\varepsilon]\le \frac{\delta eT}{\varepsilon^2 \mathbb{E}[U]},\quad \forall \varepsilon\le \varepsilon_0.
\]
\end{proposition}
\begin{proof} For the result in 1., by using both the Markov property \eqref{eq:markov} and the condition concerning the Laplace transform of $U$, we obtain:
\[
\mathbb{P}(N_T^\varepsilon>j)\le e^{\lambda T}\left(\frac{C }{\varepsilon^{2\kappa}\lambda^\kappa}\right)^j,\quad \forall \lambda>0.
\]
By choosing the optimal value of $\lambda$ given by $\lambda=\frac{j\kappa}{T}$ we obtain \eqref{eq:upper-prob}. \\
For the result in 2., we can also remark that \eqref{eq:markov} leads to
\begin{align}\label{eq:eqbound}
\mathbb{E}[N_T^\varepsilon]=\sum_{j\ge 0}\mathbb{P}(N_T^\varepsilon> j)\le \frac{e^{\lambda T}}{1-\mathbb{E}[e^{-\lambda U_1^\varepsilon}]}.
\end{align}
If 
\(
\mathbb{E}[U^2]<\infty,
\)
 we get
\[
\mathbb{E}[e^{-\lambda U_1^\varepsilon}]=1-\lambda \varepsilon^2\mathbb{E}[U]+o(\lambda\varepsilon^2),\quad \lambda>0.
\]
The particular choice $\lambda=1/T$ implies the announced result.
%gives the result: for any $\delta>1$ one can find an $\varepsilon_0 >0$ such that for any $\varepsilon \le \varepsilon_0$ 
%\[
%\mathbb{E}[N_T^\varepsilon]\le \frac{\delta eT}{\varepsilon^2 \mathbb{E}[U]}.
%\]
\end{proof}
\begin{rem} 
If the condition \(
\mathbb{E}[U^2]<\infty,
\) is not satisfied, we can construct another approach which leads to less sharp bounds. Indeed if $N^\varepsilon_T$ denotes the number of r.v. $(U_n^\varepsilon)$ such that $s_{N^\varepsilon_T}\ge T$, then the strong Markov property implies
\[
\mathbb{E}[N_T^\varepsilon]\le k\mathbb{E}[N^\varepsilon_{T/k}], \quad \forall k\in\mathbb{N}. 
\]
Taking $\lambda=k/T$ in \eqref{eq:eqbound} and afterwards $k=\lfloor T/\varepsilon^2\rfloor$ we obtain
\[
\mathbb{E}[N_T^\varepsilon]\le \frac{ke}{1-\mathbb{E}[e^{-k\varepsilon^2 U/T}]} =\frac{\lfloor T/\varepsilon^2\rfloor e}{1-\mathbb{E}[e^{-\lfloor T/\varepsilon^2\rfloor\varepsilon^2 U/T}]}\sim \frac{eT}{\varepsilon^2 (1-\mathbb{E}[e^{-U}])}\ \mbox{as}\ \varepsilon\to 0.
\]
This result is less sharp than the statement of Proposition \ref{lem:lem2} since $1-\mathbb{E}[e^{-U}]\le \mathbb{E}[U]$ but it holds even if the second moment of $U$ is not finite.
\end{rem}
Let us just mention that the large deviations theory cannot lead to interesting bounds in our case. Indeed the rate function $I$ used in Cramer's theorem satisfies:
\[
\limsup_{n\to\infty}n\ln\mathbb{P}(N_T^\varepsilon>n)\le \limsup_{n\to\infty}n\ln\mathbb{P}(s^\varepsilon_n\le T)=-\inf_{x\in[0,T]}I(x)=-\infty.
\]

\section{Approximation of one-dimensional Brownian paths
\label{sec:brow}
}%
We recall that for our approach it is essential to find a function $\phi_\varepsilon$ with compact support $[0,r_\varepsilon]$ which satisfies $\sup_{t\in[0,r_\varepsilon]}\phi_\varepsilon(t)=\varepsilon$ and such that the exit time $s_1^\varepsilon$ of the $\phi_\varepsilon$-domain is simple to generate. 

The choice of $\phi_\varepsilon$ is directly related to the method of images described by Lerche \cite{lerche-1986} and to the heat equation on some particular domain, called heat-balls and defined in Evans, Section 2.3.2, \cite{evans-2010}. More recent results on this subject can be found in \cite{deaconu-herrmann-2013}, \cite{deaconu-herrmann-2017-2} and \cite{deaconu-herrmann-2017-1}.
\begin{framed}
\centerline{\sc Brownian Skeleton $(BS)_\eta$}

\vspace*{0.2cm}
\begin{enumerate}
\item Let $\varepsilon>0$. We define 
$\phi_\varepsilon(t):=\sqrt{t\ln(\varepsilon^2 e/t)}$, for $t\in I_\varepsilon:=[0,r_\varepsilon]$ 
with $r_\varepsilon=e\varepsilon^2$.
\item Let $(A_n)_{n\ge 1}$  be a sequence of independent random variables with gamma distribution ${\Gamma}(3/2,2)$
\item Let $(Z_n)_{n\geq 1}$ be a sequence of i.i.d. Rademacher random variables (taking values +1 or -1 with probability 1/2). The sequences  $(A_n)_{n\ge 1}$ and $(Z_n)_{n\geq 1}$ are independent.
\end{enumerate}

\noindent
{\bf  Definition:} For $\varepsilon>0$ and for any  function $\eta:\mathbb{R}\to\mathbb{R}_+$, the \emph{Brownian skeleton} $(BS)_\eta$ corresponds to 
\[
\Big((U_n^\varepsilon)_{n\ge 1},(s_n^\varepsilon)_{n\ge 1},(x_n^\varepsilon)_{n\ge 0}\Big)\quad\mbox{with}\quad \left\{
\begin{array}{l}
U_n^\varepsilon=\varepsilon^2 \eta^2(x_{n-1}^\varepsilon)\,e^{1-A_n},\quad 
 s^\varepsilon_n=\ds\sum_{k=1}^nU_k^\varepsilon,\\[18pt]
  x_{n}^\varepsilon=x_{n-1}^\varepsilon+ Z_n \,\eta(x_{n-1}^\varepsilon)\phi_\varepsilon\left(\ds\frac{U_{n}^\varepsilon}{ \eta^{2}(x_{n-1}^\varepsilon)}\right),\ \forall n\ge 1 \\
 \end{array}
 \right.
\]
and $x_0^\varepsilon=x$, where $x\in \mathbb{R}$ fixed.

\end{framed}

\begin{thm}\label{thm:BM}
Let $\varepsilon>0$ and let us consider a Brownian skeleton $({\rm BS})_{\eta}$ with $\eta\equiv 1$. Then $x^\varepsilon_t=\sum_{n\ge 0}x_n^\varepsilon1_{\{  s_n^\varepsilon \le t<s^\varepsilon_{n+1}\}}$ is an $\varepsilon$-strong approximation of the Brownian paths starting in $x$. Moreover the number of approximation points on the fixed interval $[0,T]$ satisfies:
\begin{equation}\label{eq:upper-prob1}
\mathbb{P}(N_T^\varepsilon>j)\le \left( \frac{\beta' T\omega(3/2,2,\beta' )^{\beta'}}{j\varepsilon^2} \right)^{j/\beta'},\quad \forall j\in\mathbb{N},\ \forall \beta'>2,
\end{equation}
with $\omega$ a constant defined in the appendix, \eqref{eq:bound>alpha}. Moreover, for every $\delta >1$ there exists $\varepsilon_0 >0$, such that the following upper-bound holds,
\[
\mathbb{E}[N_T^\varepsilon]\le \frac{3\sqrt{3}\delta T}{\varepsilon^2},\quad \forall \varepsilon\le \varepsilon_0.
\]
\end{thm}
\begin{proof} First we remark easily that $\sup_{t\in I_\varepsilon}\phi_\varepsilon(t)=\varepsilon$ as required. So we start the skeleton of the Brownian paths $({\rm BS})_{1}$ %     
with the starting time-space value $(0,x_0^\varepsilon=x)$. Then $(0+U_1^\varepsilon, x_0^\varepsilon+Z_1\phi_\varepsilon(U_1^\varepsilon))$ stands for the first exit time and exit location of the time-space domain originated in $(0,x_0^\varepsilon)$ whose boundary is defined by $\phi_\varepsilon$. 

The second step is like the first one, it suffices to consider the new starting point $(s_1^\varepsilon,x_1^\varepsilon):=(U_{1}^\varepsilon, x_0^\varepsilon+Z_1\phi_\varepsilon(U_1^\varepsilon))$ and so on... Using the results obtained in Lerche \cite{lerche-1986} and Deaconu - Herrmann \cite{deaconu-herrmann-2017-2} (Proposition 2.2 with $\nu=-1/2$ and $a=\varepsilon\sqrt{e\pi/2}$), we know that these exit times are distributed like exponentials of gamma random variables (see, for instance, \cite{deaconu-herrmann-2017-2} Proposition A.2). In particular, the probability density function of $U_1^\varepsilon$ satisfies:
\[
f_{U_1^\varepsilon}(t)=\frac{\phi_\varepsilon(t)}{\varepsilon\sqrt{2e\pi}t}=\frac{\sqrt{\ln(\varepsilon^2 e/t)}}{\varepsilon\sqrt{2e\pi t}}1_{I_\varepsilon}(t),  \quad\forall t\in \mathbb{R}.
\]
We deduce that $U_1^\varepsilon$ and $e\varepsilon^2 W$ defined in Lemma \ref{lem1} are identically distributed (with the parameters $\alpha=\frac{3}{2}$ and $\beta=2$). By Lemma \ref{lem1}, we get for any $\beta'>2$
\[
\mathbb{E}[e^{-\lambda U_1^\varepsilon}]\le \omega\Big(\frac{3}{2},2,\beta'\Big)\left(\frac{1}{e\varepsilon^2\lambda}\right)^{1/\beta'}.
\]
Proposition \ref{lem:lem2} permits to obtain the bounds of the number of points needed to approximate the Brownian paths on the interval $[0,T]$, as $\mathbb{E}(W)=\frac{1}{3 \sqrt{3}}$.
\end{proof}
\begin{rem}
\label{rem:generaleta}
\begin{enumerate} \item Let us just notice that for $U$ a standard uniformly distributed r.v. and $G$ a standard Gaussian r.v. independent of $U$, $W=U^2e^{-G^2}$ is random variable with the PDF presented in Lemma \ref{lem1} associated to the parameters $\alpha=\frac{3}{2}$ and $\beta=2$ (for more details see \cite{devroye}, Chapter IX.3). 
\item  A similar approach will be used in the proof of Theorem \ref{thm:General}.  In Theorem \ref{thm:BM} we obtain that for $\eta\equiv 1$, we can construct a sequence of successive points corresponding to the exit time and location of $\varepsilon$-small spheroids and belonging to the Brownian trajectory. Moreover this sequence has the same distribution as $(U_n^\varepsilon, x_n^\varepsilon)_{n\ge 1}$. This procedure can also be considered for general $\eta$: $\frac{U_{n+1}^\varepsilon}{\eta^2(x_n^\varepsilon)}$ has then the same distribution as the Brownian first exit time of the $\varepsilon$-small spheroids. Therefore for any $t\in[s_n^\varepsilon,s_{n+1}^\varepsilon]$, we get
\[
|x_0^\varepsilon+B_t-x_n^\varepsilon|\le \varepsilon\eta(x_{n}^\varepsilon),
\]
where $B$ stands for the standard Brownian motion.

\end{enumerate}
\end{rem}
We can easily improve the description of the number of approximation points. Since $(U_n^\varepsilon)_{n\ge 0}$ is a sequence of independent random variables, $(N^\varepsilon_t)_{t\ge 0}$ is a renewal process and the classical asymptotic description holds:
\begin{proposition}\label{prop:ren} We consider the Brownian skeleton $(BS)_\eta$ for $\eta=1$. 
We define, as previously
\[
N_T^\varepsilon:=\inf\{ n\ge 1:\ s^\varepsilon_n\ge T \}.
\]
the number of approximation points needed to cover the time interval $[0,T]$ for $T$ a fixed positive time. 
Then:
\[
\lim_{\varepsilon\to 0}\varepsilon^2\mathbb{E}[N^\varepsilon_T]=\frac{T}{e}\, 3^{3/2}.
\]
Moreover the following central limit theorem holds:
\[
\ds\lim_{\varepsilon \to 0} \sqrt{\frac{\mu^3}{\varepsilon^2\sigma^2 T}}\Big( \varepsilon^2 N^\varepsilon_T -\frac{T}{e}\, 3^{3/2}\Big) = G \quad \mbox{in distribution},
\]
where $G$ is a $\mathcal{N}(0,1)$ standard Gaussian random variable, $\mu=e\, 3^{-3/2}\approx 0.5231336$ and $\sigma^2=(5^{-3/2}-3^{-3})\, e^2\approx 0.3872285 $.
%where $\mu=\mathbb{E}[e^{1-A}] $, $A$ being a gamma distributed random variable $\gamma(3/2,2)$.
\end{proposition}
\begin{proof}
Let us consider $(\overline{N}_t)_{t\ge 0}$ a renewal process with interarrivals $(e^{1-A_n})_{n\ge 1}$ independent random variables defined in  Theorem \ref{thm:BM}. The interarrival time satisfies $\mathbb{E}[e^{1-A_1}]=e\mathcal{L}_{A_1}(1)$ where $\mathcal{L}_{A_1}$ stands for the Laplace transform of $A_1$. Since in our case $A_1$ is gamma distributed, it is well known that $\mathcal{L}_{A_1}(s)=(2s+1)^{-3/2}$. \\
We use here classical results for the renewal theory, see for example \cite{daley-verejones-2002}. The elementary renewal theorem leads to 
\[
\lim_{t\to\infty}\frac{\mathbb{E}[\overline{N}_{t}]}{t}=\frac{1}{\mathbb{E}[e^{1-A_1}]}=\frac{3^{3/2}}{e}.
\]
In order to obtain the first part of the statement, it suffices to observe that:
\begin{align*}
N_T^\varepsilon&=%\inf\Big\{n\ge 0: \ \varepsilon^2\sum_{k=1}^ne^{1-A_k}\ge T\Big\}=
\inf\Big\{n\ge 0: \ \sum_{k=1}^ne^{1-A_k}\ge \frac{T}{\varepsilon^2}\Big\}=\overline{N}_{T/\varepsilon^2}.
\end{align*}
We deduce that
\[
\lim_{\varepsilon\to 0}\varepsilon^2 \mathbb{E}[N_T^\varepsilon]=\lim_{\varepsilon\to 0}\varepsilon^2\mathbb{E}[\overline{N}_{T/\varepsilon^2}]=\lim_{t\to\infty}T\frac{\mathbb{E}[\overline{N}_t]}{t}=\frac{T}{e}\,3^{3/2}.
\]
The same argument holds for the CLT: if we denote by $\mu=\mathbb{E}[e^{1-A_1}]$ and $\sigma^2={\rm Var}(e^{1-A_1})$ then
\[
\ds\lim_{\varepsilon \to 0} \sqrt{\frac{t\mu^3}{\sigma^2}}\Big( \frac{\overline{N}_t}{t} -\frac{1}{\mu}\Big) = G \quad\mbox{in distribution},
\]
where $G$ is a $\mathcal{N}(0,1)$ standard Gaussian random variable. The statement is therefore a consequence of the link between $N^\varepsilon_T$ and $\overline{N}_{T/\varepsilon^2}$.
\end{proof}
%\begin{Com}
%Il est important tout au long du texte de bien faire la différence entre les variables $t$ et $T$, toutes les deux représentant le temps observé. Jusqu'ici le processus était observé sur un temps fixé à l'avance: l'intervalle $[0,T]$ et on s'intéresse donc au nombre de pas noté $N_T^\varepsilon$. Le temps $t$ n'apparaît pour l'instant que dans la preuve de la Proposition 3.2 quand il s'agit de définir un processus de renouvellement sur tout intervalle de temps, il sera donc noté $\overline{N}_t$ pour tout $t\le T$. Il est important de conserver les mêmes notations tout au long de la section 4.
%\end{Com}
%
%
%
%
%
\section{The particular L and G classes of diffusion}
\label{sec:diff}
Let us now consider some generalizations  of the Brownian paths study. 
We introduce solutions of the following one-dimensional stochastic differential equation:
\begin{equation}
\label{eq:eds}
dX_t=\sigma(t,X_t)dB_t+\mu(t,X_t)\,dt,\quad X_0=x_0,
\end{equation}
where $(B_t,\ t\ge 0)$ stands for a standard one-dimensional Brownian motion and $\sigma, \mu :[0,+\infty)\times\mathbb{R}\to\mathbb{R}$. Let us consider two families of diffusions introduced in Wang - P\"otzelberger \cite{wp}:
\begin{enumerate}
\item ($L$-class) for $\sigma(t,x)=\overline{\sigma}(t)$ and $\mu(t,x)=a(t)x+b(t)$, $x\in\mathbb{R}$
\item ($G$-class) for $\sigma(t,x)=\underline{\sigma}x$ and $\mu(t,x)=a(t)x+b(t)x\ln(x)$, $x\in\mathbb{R}_+$,
\end{enumerate}
where $\overline{\sigma}: \mathbb{R}_+\to\mathbb{R}_+ , a, b : \mathbb{R}_+\to\mathbb{R}$ are $\mathcal{C}^1$-functions and $\underline{\sigma}\in\mathbb{R}_+$.

Let us note that, in such particular cases, the solution of the SDE \eqref{eq:eds}  has the same distribution as a function of the time-changed Brownian motion:
\begin{equation}
\label{eq-f}
X_t=f(t,x_0+B_{\rho(t)}),\quad t\ge 0,
\end{equation}
 (where $f$ and $\rho$ denote functions that we specify for each class afterwards).
 
For $L$-class diffusions for instance one particular choice of the function $f$ (this choice is not unique) is given by (see, for instance, Karatzas and Shreve \cite{karatzas-shreve-1991}, p. 354, Section 5.6 for classical formulas and Herrmann and  Massin \cite{herrmann2020approximation} for new developments in this topic):
\begin{equation}
\label{def-f}
f(t,x)=\frac{\overline{\sigma}(t)}{\sqrt{\rho'(t)}}\, x+c(t),
\end{equation}
with % $c(0)=0$,
\[
c(t)=e^{\int_0^t a(s)\,ds}\int_0^t b(s)e^{-\int_0^s a(u)\,du}\,ds,\quad\mbox{and}\quad \rho(t)=\int_0^t \overline{\sigma}^2(s)e^{-2\int_0^s a(u)\,du}\,ds.
\]

\begin{rem}
\label{rem-L-G} 
If we have a diffusion in the $L$-class characterized by some fixed function $f(t,x)$ given in \eqref{eq-f} then we can obtain a diffusion of $G$-class by using the function $e^{f(t,x)}$ instead of $f(t,x)$. Obviously the corresponding coefficients $a,b,$ and $\underline{\sigma}$ need to be specified with respect to those connected to $f(t,x)$.

\end{rem}
\begin{proposition} Let us define the following diffusion process
\begin{equation}
X_t=f(t,x_0+B_{\rho(t)}),\quad t\ge 0,
\end{equation}
where $f$ is given by \eqref{def-f}. Then $X_t$ is a weak solution of the stochastic differential equation \eqref{eq:eds}.
\end{proposition}
\begin{proof}  We can write the previous expression for $f$  on the form 
\begin{equation}
f(t,x) = x \cdot e^{\int_0^t a(s) \dint s}+e ^{\int_0^t a(s) \dint s}\ds\int_0^t b(s) e^{-\int_0^s a(u)\dint u}\dint s.
\end{equation}
We denote also by
\begin{equation}
\mathbb{I}_t =\ds\int_0^t \sqrt  {\rho ' (s) }\dint B_s.
\end{equation}
In order to prove the result we need to prove that 
$X_t = f(t, x_0 +\mathbb{I}_t) $ 
satisfies the equation \eqref{eq:eds}. 
For the initial condition we can see that:
\begin{equation}
X_0=f(0,x_0) = x_0.
\end{equation}
Let us now evaluate
\begin{equation}
\begin{array}{ll}
\dint X_t& =  \left[ a(t) e^{\int_0^t a(s) \dint s} (x_0+\mathbb{I}_t) + a(t) e^{\int_0^t a(s) \dint s} \ds\int_0^t b(s) e^{-\int_0^s a(u) \dint u}\dint s + e^{\int_0^t a(s)\dint s} b(t)e^{-\int_0^t a(s)\dint s} \right] \dint t \\
& \qquad  + \overline{\sigma} (t)\dint B_t\\
&= \left[ a(t) f(t, x_0+\mathbb{I}_t) + b(t)\right ] \dint t +\overline {\sigma}(t) \dint B_t = \left[ a(t) X_t +b(t) \right] \dint t +\overline{\sigma} \dint B_t,
\end{array}
\end{equation}
by using It\^o formula.\\
This ends the proof of the proposition.
\end{proof}
In this section, we consider particular diffusion processes which are strongly related to the Brownian paths. It is therefore intuitive to replace in \eqref{eq-f} the Brownian trajectory by its approximation. If the function $f$ is Lipschitz continuous, then the error stemmed from the approximation is easily controlled (the proof is left to the reader). 
\begin{assu} \label{assu} The diffusion process $(X_t)_{t\ge 0}$ satisfies
\[
X_t=f(t,x_0+B_{\rho(t)})
\]
with $f$ a Lipschitz continuous function:
\begin{equation}
\label{Lip}
|f(t,x)-f(s,y)|\le K_{Lip}(T)(|x-y|+|t-s|),\quad  \forall (x,y)\in\mathbb{R}^2,\quad \forall (s,t)\in [0,T]^2
\end{equation}
where $K_{Lip}(T)$ stands for the Lipschitz constant. The function $\rho$ is a strictly increasing continuous function with initial value $\rho(0)=0$. 
\end{assu}
%%Let us note that this assumption is satisfied for $L$-class diffusions as soon as $\overline{\sigma}$ is positive and all the functions $\overline{\sigma}$, $a$ and $b$ are continuous {\color{blue} (en fait ce n'est pas vrai....)}
\begin{proposition} Consider $T>0$ and $\varepsilon >0$ fixed. Let the diffusion process $(X_t)$ satisfy Assumption \ref{assu} and let 
$$x_t^{\theta}:=\sum_{n\ge 0}x_n^{\theta}1_{\{ s_n^{\theta}\le t <s_{n+1}^{\theta} \}}$$ be a $\theta$-strong approximation of the Brownian motion (see Theorem \ref{thm:BM}) with $\theta=\varepsilon K_{Lip}^{-1}(\rho^{-1}(T))$ on the time interval $[0,\rho^{-1}(T)]$, where $K_{Lip}(T)$ is defined in \eqref{Lip}, then 
\[
y^\varepsilon_t:=\sum_{n\ge 0}f(\rho^{-1}(s_n^\theta),x_n^\theta)1_{\{  s_n^{\theta}\le \rho(t) <s_{n+1}^{\theta}\}}
\]
is an $\varepsilon$-strong approximation of $(X_t)$ on $[0,T]$.
\end{proposition}
Unfortunately the Lipschitz continuity of the function $f$ is a restrictive condition which is not relevant for most of the diffusion processes. In particular, a typical diffusion belonging to the $L$ or $G$-class does not satisfy the Lipschitz condition. Consequently we introduce a more general framework.
\begin{assu} \label{assu1}
The diffusion process $(X_t)_{t\ge 0}$ is a function of the time-changed Brownian motion: 
\[
X_t=f(t,x_0+B_{\rho(t)})
\]
where $\rho$  is an increasing differentiable function satisfying $\lim_{t\to\infty}\rho(t)=\infty$ with initial value $\rho(0)=0$,
$f$ is a $\mathcal{C}^{1,1}(\R_+\times\R,\R)$-function. There exists a strictly increasing  $\mathcal{C}^2(\mathbb{R})$-function $F$ such that
\begin{equation}
\label{eq:cond1}
\sup_{t\in[0,T]}\max\Big\{ \left| \frac{\partial f}{\partial t}(t,x) \right| , \left| \frac{\partial f}{\partial x}(t,x) \right| \Big\} \le F(x^2),\quad \forall x\in\mathbb{R}.
\end{equation} 
Furthermore we assume:
\begin{itemize} 
\item[-] there exist two constants $\kappa_1$ and $\kappa_2$ such that $\max(F,F',F'')(x^2)\le \kappa_1 e^{\kappa_2 x}$ for all $x\ge 0$
\item[-] the function $x\mapsto \frac{1}{F(2x^2+1)}$ is Lipschitz-continuous.
\end{itemize} \end{assu}
\begin{assu} \label{assu2}
 $\exists\kappa_{\min}>0$ such that $\rho'(\rho^{-1}(x))\ge \kappa_{\rm min}$ for all $x\in\mathbb{R}$.
\end{assu}

\begin{rem} One can check easily that the L and $G$-class diffusions verify these hypotheses.
\end{rem}
Let us define the function $\eta$ to be
\begin{equation}
\label{def:eta1}
\eta(x)=\frac{1}{(e\kappa_{\rm min}^{-1}+1)F(2x^2+1)}.
\end{equation}
By Assumption \ref{assu1} this function is strictly decreasing and Lipschitz continuous.
%We define also, for $\varepsilon >0$ the function  
%\begin{equation}
%\label{def:eta}
%\eta_\varepsilon(x)=\varepsilon\eta(x).
%\end{equation} 

%\begin{Com} Il faudrait écrire une remarque ici pour montrer que les diffusions de classe L satisfont cette propriété. A moins que ce soit plus judicieux d'écrire un théorème en fin de section. Il faut juste réfléchir comment structurer la section car on commence à introduire les diffusions de classe L et puis, après, on ne les utilise plus directement.
%!TEX encoding = UTF-8 Unicode\end{Com}

%Let us consider the following approximation of the Brownian paths: 
%\begin{equation}\label{eq:approx-BM}
%\left\{ \begin{array}{l}
%U_n^\varepsilon=\eta_\varepsilon^2(x_{n-1})e^{1-A_n},\\[5pt] 
%s_n^\varepsilon=\sum_{k=1}^nU_k\\[5pt] 
%x_n^\varepsilon=x_{n-1}+\eta_\varepsilon(x_{n-1}) Z_n \phi(\eta_{\varepsilon}^{-2}(x_{n-1})U_n).
%\end{array}\right.
%\end{equation}
%Here $\phi$ stands for $\phi_\varepsilon$ with $\varepsilon=1$ defined in the statement of Theorem \ref{thm:BM}.\\
%%endcolor
%For notational simplicity we omit the parameter $\varepsilon$ in the notation of $x_n^\varepsilon$, $U_n^\varepsilon$, $s_n^\varepsilon$.
\begin{thm} \label{thm:General} Let $\varepsilon>0$. Let $(X_t)_{t\ge 0}$ be the solution of \eqref{eq:eds} satisfying Assumptions \ref{assu1} and \ref{assu2} and let us consider the Brownian skeleton $({\rm BS})_\eta$ associated to the function $\eta$ defined in \eqref{def:eta1}, then 
\begin{equation}\label{eq:approx:diff}
y^\varepsilon_t:=\sum_{n\ge 0}f(\rho^{-1}(s_n^\varepsilon),x_n^\varepsilon)1_{\{  s_n^\varepsilon\le \rho(t) <s_{n+1}^\varepsilon\}}
\end{equation}
 is an $\varepsilon$-strong approximation of $(X_t)$ on $[0,T]$.
\end{thm}
\begin{proof}
Let us assume that $t$ satisfies $s_{n}^\varepsilon\le \rho(t) \le s_{n+1}^\varepsilon$ for some $n\in\mathbb{N}$. We denote $t_n^\varepsilon:=\rho^{-1}(s_n^\varepsilon)$ and $A^n_t:=f(t,x_0^\varepsilon+B_{\rho(t)})-f(t_{n}^\varepsilon,x_{n}^\varepsilon)$. We obtain, there exists $\tau \in (0,1)$ such that:
\begin{align*}
A^n_t&=(t-t_{n}^\varepsilon) \frac{\partial f}{\partial t}(t_n^\varepsilon+\tau (t-t_n^\varepsilon),x_n^\varepsilon+\tau (x_0^\varepsilon+B_{\rho(t)}-x_n^\varepsilon))\\
&\quad +(x_0^\varepsilon+B_{\rho(t)}-x_n^\varepsilon)\frac{\partial f}{\partial x}(t_n^\varepsilon+\tau (t-t_n^\varepsilon),x_n^\varepsilon+\tau (x_0^\varepsilon+B_{\rho(t)}-x_n^\varepsilon)).
\end{align*}
Under the assumption \eqref{eq:cond1}
we have
\begin{align*}
|A^n_t|&\le \Big(|t-t_n^\varepsilon|+|x_0^\varepsilon+B_{\rho(t)}-x_n^\varepsilon|\Big)\cdot F((x_n^\varepsilon+\tau (x_0^\varepsilon+B_{\rho(t)}-x_n^\varepsilon))^2).
\end{align*}
Since 
\[
|t-t_n^\varepsilon|\le |t_{n+1}^\varepsilon-t_n^\varepsilon|=|\rho^{-1}(s_{n+1}^\varepsilon)-\rho^{-1}(s_{n}^\varepsilon)|=\Big|  \int_{s_n^\varepsilon}^{s_{n+1}^\varepsilon}\frac{\dint u}{\rho'(\rho^{-1}(u))}\Big|,
\]
we obtain
\[
|t-t_n^\varepsilon|\le \kappa_{\rm min}^{-1} |s_{n+1}^\varepsilon-s_n^\varepsilon|=\kappa_{\rm min}^{-1} |U_{n+1}^\varepsilon|\le e\, \kappa_{\rm min}^{-1}  \varepsilon^2\eta^2(x_{n}^\varepsilon).
\]
Moreover, by the definition of the BM approximation (see Remark \ref{rem:generaleta}),
\[
|x_0^\varepsilon+B_{\rho(t)}-x_n^\varepsilon|\le  \varepsilon\eta(x_{n}^\varepsilon).
\]
Finally due to the monotone property of $F$, 
\[
\begin{array}{ll}
|A^n_t| & \le \Big( e\, \kappa_{\rm min}^{-1} \varepsilon^2 \eta^2(x_{n}^\varepsilon)+\varepsilon\eta(x_{n}^\varepsilon)\Big)\cdot F((x_n^\varepsilon+\tau (x_0^\varepsilon+B_{\rho(t)}-x_n^\varepsilon))^2)\\
& \le \Big( e\, \kappa_{\rm min}^{-1}  \varepsilon^2\eta^2(x_{n}^\varepsilon)+\varepsilon\eta(x_{n}^\varepsilon)\Big)\cdot F(2(x_n^\varepsilon)^2+2\varepsilon^2\eta^2(x_n^\varepsilon)).\\
\end{array}
\]
There exists $\varepsilon_0>0$ such that $\varepsilon^2\eta^2(x)\le 1/2$ for all $x\in\mathbb{R}$ and $\varepsilon\le \varepsilon_0$. Then, by the definition of the function $\eta$,  for $\varepsilon\le \varepsilon_0$, we have
\[
|A^n_t|\le (e\, \kappa_{\rm min}^{-1}  +1 )\varepsilon\eta(x_n^\varepsilon)\cdot F(2(x_n^\varepsilon)^2+1)\le \varepsilon,
\]
for any $t\in [s_n^\varepsilon,s_{n+1}^\varepsilon]$. We deduce that the piecewise constant approximation associated to $(y_n^\varepsilon)_n$ where $y_n^\varepsilon:=f(\rho^{-1}(s_n^\varepsilon),x_n^\varepsilon)$ and $n\le \inf\{k\ge 0:\ s_k\ge \rho(T)\}$ is a $\varepsilon$-strong approximation of $(X_t,\, t\in [0,T])$.
\end{proof}

Let us now describe the efficiency of the $\varepsilon$-strong approximation. We introduce
\begin{equation}\label{eq:def:count}
N_t^\varepsilon:=\inf\{n\ge 0: \ s_n^\varepsilon\ge t\}\quad \mbox{and}\quad \hat{N}_t^\varepsilon:=N_{\rho(t)}^\varepsilon,
\end{equation}
where $s_n^\varepsilon$ is issued from the Brownian skeleton $({\rm BS})_{\eta}$. $\hat{N}_t^\varepsilon$ corresponds therefore to the number of random points needed to approximate the diffusion paths on $[0,t]$. Let us observe that the random variables $U_n^\varepsilon$ are no more i.i.d. random variables in the diffusion case (different to the Brownian case), therefore we cannot use the classical renewal theorem in order to describe  $\hat{N}_t^\varepsilon$.
%\begin{Com} Faut-il ici tout définir juste avec la variable $t$ ou plutôt à la fois avec $t$ et $T$ pour rappeler que notre intérêt est la description du nombre de pas sur l'intervalle $[0,T]$ ?
%\end{Com}
{
\begin{proposition}\label{prop:finite-mean}
Let $(N^\varepsilon_t,\, t\ge 0)$ be the counting process defined by \eqref{eq:def:count}. Then, under Assumptions \ref{assu1} and \ref{assu2}, for any $x\in\mathbb{R}$ and $t\ge 0$, the average $\psi^\varepsilon(t,x):=\mathbb{E}[N^\varepsilon_t|x_0^\varepsilon=x] $ is finite. Moreover there exists a constant $\lambda_0>0$ such that the Laplace transform $\mathcal{L}\psi^\varepsilon(\lambda,x):=\int_0^\infty e^{-\lambda t}\psi^\varepsilon(t,x)\,dt$ is finite for any $\lambda\in\mathbb{C}$ satisfying ${\rm Re}(\lambda)>\lambda_0$.
\end{proposition}
\begin{proof} The mean of the counting process is defined by
\[
\psi^\varepsilon(t,x)=\sum_{n\ge 1}\mathbb{P}(N_t^\varepsilon\ge n)=\sum_{n\ge 1}\mathbb{P}(s_n^\varepsilon\le t).
\]
This equality holds since $s_n^\varepsilon$ is a continuous random variables by the definition of $({\rm BS})_{\eta}$.
Let us denote by $m_n^\varepsilon:=\min_{0\le k\le n-1}\eta^2(x_k^\varepsilon)$, where $\eta$ is defined by \eqref{def:eta1} and introduce the following decomposition:
\begin{equation}\label{eq:decomp}
\mathbb{P}(s_n^\varepsilon\le t) =  \mathbb{P}(s_n^\varepsilon\le t,\, m_n^\varepsilon\ge n^{-2/3})+\mathbb{P}(s_n^\varepsilon\le t, \,m_n^\varepsilon<n^{-2/3}).
\end{equation}
By the definition of the sequence $(s_n^\varepsilon)$, we get
\[
s_n^\varepsilon=\varepsilon ^2\eta^2(x_0^\varepsilon)e^{1-A_1}+\ldots+\varepsilon ^2\eta^2(x_{n-1}^\varepsilon)\,e^{1-A_n}\ge \varepsilon ^2 m_n^\varepsilon(e^{1-A_1}+\ldots+e^{1-A_n}).
\]
Hence, for any $\lambda>0$, we have
\begin{align}\label{eq:upper1}
 u_n^\varepsilon(t):=\mathbb{P}(s_n^\varepsilon\le t,\, m_n^\varepsilon\ge n^{-2/3} )&\le \mathbb{P}(e^{1-A_1}+\ldots+e^{1-A_n}\le tn^{2/3}\varepsilon ^{-2})\nonumber\\
 &=\mathbb{P}\left(\exp\left\{-\frac{\lambda}{2} n^{-2/3}\varepsilon^2 (e^{1-A_1}+\ldots+e^{1-A_n})\right\}\ge e^{- \frac{\lambda t}{2}}\right)\nonumber\\
 &\le e^{ \frac{\lambda t}{2}}\mathbb{E}\Big[\exp\left(-\frac{\lambda}{2}\,n^{-2/3} \varepsilon^2 e^{1-A_1}\right)\Big]^n.
\end{align}
Since the second moment of $e^{1-A_1}$ is finite, we obtain the Taylor expansion:
\[
\mathbb{E}\left[\exp\left(-\frac{\lambda}{2} n^{-2/3} \varepsilon^2 e^{1-A_1}\right)\right]=1-\frac{\lambda}{2} \frac{n^{-2/3}\varepsilon^2 e}{3^{3/2}}+\left(\frac{\lambda}{2}\right)^2 \frac{n^{-4/3}\varepsilon^4 e^2}{2\cdot 5^{3/2}}+o(n^{-4/3}).
\]
By using the classical relation $\ln(1-x) =-\left( x+\frac{x^2}{2}+\frac{x^3}{3}+\ldots\right)$, for $x\in (0,1)$, we can deduce that 
\begin{equation}
\mathbb{E}\left[\exp\left(-\ds\frac{\lambda}{2} n^{-2/3}\varepsilon^2 e^{1-A_1}\right)\right]^n=\exp\left[ -\ds\frac{\lambda}{2} \ds\frac{e n^{1/3}\varepsilon^2 }{3^{3/2}}\right]+o(1) \mbox{ as  } n\to +\infty.
\end{equation}
Using classical results on series with positive terms, we obtain by comparison 
\begin{equation}
\label{eq:premier}
\sum_{n\ge 1} \mathbb{P}(s_n^\varepsilon\le t,\, m_n^\varepsilon\ge n^{-2/3} )<\infty.
\end{equation}
Let us just note that this result is still true if we consider the terms $\overline{u}_n^\varepsilon(\lambda):=\int_0^\infty e^{-\lambda t} u_n^\varepsilon(t)\,dt$. Indeed \eqref{eq:upper1} leads to
\[
\overline{u}_n^\varepsilon(\lambda)\le \int_0^\infty e^{-\lambda t} e^{ \frac{\lambda t}{2}}\mathbb{E}\left[\exp\left(-\ds\frac{\lambda}{2}  n^{-2/3} \varepsilon^2 e^{1-A_1}\right)\right]^n\,\dint t=\frac{2}{\lambda}\,\mathbb{E}\left[\exp\left(- \ds\frac{\lambda}{2} n^{-2/3} \varepsilon^2e^{1-A_1}\right)\right]^n.
\]
Since the upper bound is the term of a convergent series, we deduce by comparison that 
\begin{equation}
\label{eq:premm}
\sum_{n\ge 1}\overline{u}_n^\varepsilon(\lambda)<\infty,\quad \forall \lambda>0.
\end{equation}

Let us now focus our attention to the second term of the r.h.s in \eqref{eq:decomp}.  Since we consider the Brownian skeleton $({\rm BS})_{\eta}$, the sequence $(s_n^\varepsilon, x_n^\varepsilon)$  belongs to the graph of a Brownian trajectory (see Remark \ref{rem:generaleta}). Consequently the condition $m_n^\varepsilon<n^{-2/3}$ can be related to a condition on the Brownian paths:
\[
v_n^\varepsilon(t):=\mathbb{P}(s_n^\varepsilon\le t, \,m_n^\varepsilon<n^{-2/3})\le \mathbb{P}(\exists s\le t\ {\rm s.t.}\ \eta^2(x_0^\varepsilon+B_s)<n^{-2/3}),
\]
where $B$ is a one-dimensional standard Brownian motion. Using the upper-bound of the function $F$ in Assumption \ref{assu1} and the definition of the function $\eta$, we obtain the bound: there exists $C>0$ and $\kappa>0$ such that $\eta(x)\ge C e^{-\kappa|x|}$. Let us note that $x_0^\varepsilon=x$. The Brownian reflection principle leads to
\begin{align*}
v_n^\varepsilon(t)&\le \mathbb{P}\left(\eta^2\left(\sup_{s\in[0,t]}|x+B_s|\right)<n^{-2/3}\right)\\
&\le \mathbb{P}\Big(\sup_{s\in[0,t]}|x+B_s|>\frac{1}{3\kappa}\,\ln(n C^{3})\Big)\\
&\le \mathbb{P}\Big(\sup_{s\in[0,t]}|B_s|>\frac{1}{3\kappa}\,\ln(n C^{3})-|x|\Big)\\
&\le 2\mathbb{P}\Big(|B_t|>\frac{1}{3\kappa}\,\ln(n C^{3}) -|x| \Big)\\
& \le \frac{12\kappa\sqrt{t}}{\ln(n C^{3}/e^{3\kappa |x|})\sqrt{2\pi}}\, \exp\Big(-\frac{\ln^2(n C^{3}/e^{3\kappa |x|})}{72\kappa^2t}\Big).
\end{align*}
Since $\frac{\ln^2(n C^{3}/e^{3\kappa |x|})}{72\kappa^2t}\ge 2\ln(n)$ for large values of $n$, we deduce that the r.h.s of the previous equality corresponds to the term of a convergent series. Therefore
 \begin{equation}
 \label{sec}
 \sum_{n\ge 1} \mathbb{P}(s_n^\varepsilon\le t, \,m_n^\varepsilon<n^{-2/3})<\infty.
 \end{equation}
This finishes the proof of finiteness of $\psi^\varepsilon(t,x)$.\\ 
 Let us now define $\overline{v}_n^\varepsilon(\lambda):=\int_0^\infty e^{-\lambda t} v_n^\varepsilon(t)\,\dint t$. The previous inequalities permit to obtain:
 \begin{align*}
 \overline{v}_n^\varepsilon(\lambda)\le \int_0^\infty e^{-\lambda t}  2\mathbb{P}\Big(|B_t|>\frac{1}{3\kappa}\,\ln(n C^{3}/e^{3\kappa|x|})  \Big)\,\dint t=\sqrt{\frac{2}{\pi}}\iint_{\mathbb{R}_+^2}e^{-\lambda t-\frac{u^2}{2}}1_{\{ u\ge \alpha_n/\sqrt{t}\}}\,\dint t\, \dint u
 \end{align*}
 where $\alpha_n=\frac{1}{3\kappa}\,\ln(n C^{3}/e^{3\kappa |x|}) $. Hence
 \begin{align*}
 \overline{v}_n^\varepsilon(\lambda)\le\sqrt{\frac{2}{\pi}}\frac{1}{\lambda} \int_0^\infty e^{-\frac{\lambda \alpha_n^2}{u^2}-\frac{u^2}{2}}\, \dint u.
 \end{align*}
 By the change of variable $u=\sqrt{r}(2\lambda\alpha_n^2)^{1/4}$, we have
  \begin{align*}
 \overline{v}_n^\varepsilon(\lambda)&\le \frac{(2\lambda\alpha_n^2)^{1/4}}{\lambda\sqrt{2\pi}} \int_0^\infty \frac{1}{\sqrt{r}}\, e^{-\alpha_n\sqrt{\frac{\lambda}{2}}(\frac{1}{r}+r)}\, \dint r\\
 &=\frac{(2\lambda\alpha_n^2)^{1/4}}{\lambda\sqrt{2\pi}} \Big\{\int_0^1 \frac{1}{\sqrt{r}}\, e^{-\alpha_n\sqrt{\frac{\lambda}{2}}(\frac{1}{r}+r)}\, \dint r+\int_1^\infty \frac{1}{\sqrt{r}}\, e^{-\alpha_n\sqrt{\frac{\lambda}{2}}(\frac{1}{r}+r)}\, \dint r\Big\}.
 \end{align*}
 Using the change of variable $r\mapsto\frac{1}{r}$ in the first integral leads to
  \begin{align*}
 \overline{v}_n^\varepsilon(\lambda)&\le \frac{(2\lambda\alpha_n^2)^{1/4}}{\lambda\sqrt{2\pi}} \int_1^\infty \Big(\frac{1}{\sqrt{r}}+\frac{1}{r^{3/2}}\Big)\, e^{-\alpha_n\sqrt{\frac{\lambda}{2}}(\frac{1}{r}+r)}\, \dint r\le \frac{2(2\lambda\alpha_n^2)^{1/4}}{\lambda\sqrt{2\pi}}\int_1^\infty \, e^{-\alpha_n\sqrt{\frac{\lambda}{2}}r}\, \dint r\\
 &\le \frac{1}{\sqrt{\pi \alpha_n}}\,\Big( \frac{2}{\lambda} \Big)^{5/4}\, e^{-\alpha_n\sqrt{\frac{\lambda}{2}}}.
 \end{align*}
Since $\alpha_n\sim \frac{1}{3\kappa}\,\ln n$, the upper bound is  a term of a convergent series as soon as $\lambda>\lambda_0:=18\kappa^2$. Therefore, by comparison,
\begin{equation}
\label{eq:secundo}
\sum_{n\ge 1}\overline{v}_n^\varepsilon(\lambda)<\infty\quad \mbox{for}\quad \lambda>\lambda_0.
\end{equation}
Combining \eqref{eq:decomp}, \eqref{eq:premier} and \eqref{sec} leads to the announced statement $\psi^\varepsilon(t,x)<\infty$. Since 
\[
\int_0^\infty e^{-\lambda t}\psi^\varepsilon(t,x)\,\dint t=\sum_{n\ge 1}\overline{u}_n^\varepsilon(\lambda)+\sum_{n\ge 1}\overline{v}_n^\varepsilon(\lambda),
\]
the convergence \eqref{eq:premm} and \eqref{eq:secundo} of both series for $\lambda>\lambda_0$ implies that the Laplace transform is well defined for $\lambda>\lambda_0$. Of course, this result can be extended for complex values $\lambda\in\mathbb{C}$ satisfying ${\rm Re}(\lambda)>\lambda_0$.
\end{proof}
\begin{proposition}\label{prop:conti} Under Assumption \ref{assu2}, the function $(t,x)\mapsto \psi^\varepsilon(t,x)=\mathbb{E}[N^\varepsilon_t|x_0^\varepsilon=x]$ is continuous.
\end{proposition}
\begin{proof}
Let us consider the Brownian skeleton $({\rm BS})_\eta$ which corresponds to the sequences $(U_n^\varepsilon)_{n\ge 1}$, $(s_n^\varepsilon)_{n\ge 1}$ and $(x_n^\varepsilon)_{n\ge 0}$ with $x_0^\varepsilon=x$. We consider also a second Brownian approximation $(\hat{U}_n^\varepsilon)_{n\ge 1}$, $(\hat{s}_n^\varepsilon)_{n\ge 1}$ and $(\hat{x}_n^\varepsilon)_{n\ge 0}$ with $\hat{x}_0^\varepsilon=\hat{x}$, both approximations being constructed with respect to the same r.v. $(A_n)_{n\ge 1}$,  $(Z_n)_{n\ge 1}$ and the same function $\eta$. The corresponding counting processes are denoted by $N_t^\varepsilon$ and $\hat{N}_t^\varepsilon$. \\
{\bf Step 1.} Let us describe the distance between these two schemes. The function $\eta$ is bounded so we denote by $M=\sup_{x\in\mathbb{R}}\eta(x)$ and $L_{\rm Lip}$ the Lipschitz constant of $\eta$. Hence
\[
|\eta^2(x_n^\varepsilon)-\eta^2(\hat{x}_n^\varepsilon)|\le 2ML_{\rm Lip}|x_n^\varepsilon-\hat{x}_n^\varepsilon|,\quad \forall n\ge 0.
\]
Using the definition of the approximations $({\rm BS})_{\eta}$, we have
\begin{align*}
|x_n^\varepsilon-\hat{x}_n^\varepsilon|&=|x_{n-1}^\varepsilon-\hat{x}_{n-1}^\varepsilon+\varepsilon Z_n\phi_1(e^{1-A_n})(\eta(x_{n-1}^\varepsilon)-\eta(\hat{x}_{n-1}^\varepsilon))|\nonumber \\
&\le |x_{n-1}^\varepsilon-\hat{x}_{n-1}^\varepsilon|+\varepsilon |\eta(x_{n-1}^\varepsilon)-\eta(x_{n-1}^\varepsilon)|\nonumber\\
&\le (1+\varepsilon L_{\rm Lip}) |x_{n-1}^\varepsilon-\hat{x}_{n-1}^\varepsilon|\le  (1+\varepsilon L_{\rm Lip})^n |x_{0}^\varepsilon-\hat{x}_{0}^\varepsilon|.
\end{align*}
We deduce
\begin{equation}
\label{eq:lipsch}
\max_{0\le k\le n}|\eta^2(x_k^\varepsilon)-\eta^2(\hat{x}_k^\varepsilon)|\le 2ML_{\rm Lip}  (1+\varepsilon L_{\rm Lip})^n |x-\hat{x}|.
\end{equation}
{\bf Step 2.} Since $N_t^\varepsilon$ is a $\mathbb{N}$-valued random variable, we get
 \begin{align*}
\psi^\varepsilon(t,x)&=\sum_{n\ge 1}\mathbb{P}(N_t^\varepsilon\ge n)=\sum_{n\ge 1}\mathbb{P}(s_n^\varepsilon\le t)=\sum_{n\ge 1}\mathbb{P}(s_n^\varepsilon\le t,\, \hat{s}_n^\varepsilon\le t)+\sum_{n\ge 1}\mathbb{P}(s_n^\varepsilon\le t,\,\hat{s}_n^\varepsilon>t)\\
&=\sum_{n\ge 1}\mathbb{P}(\hat{s}_n^\varepsilon\le t)-\sum_{n\ge 1}\mathbb{P}(\hat{s}_n^\varepsilon\le t,\,s_n^\varepsilon>t)+\sum_{n\ge 1}\mathbb{P}(s_n^\varepsilon\le t,\,\hat{s}_n^\varepsilon>t).
\end{align*}
We deduce that
\begin{equation}
\label{eq:bound-diff}
| \psi^\varepsilon(t,x)-\psi^\varepsilon(t,\hat{x}) |\le \sum_{n\ge 1}u_n^\varepsilon(t)+\sum_{n\ge 1}\overline{u}_n^\varepsilon(t),
\end{equation}
where $u_n^\varepsilon(t)=\mathbb{P}(\hat{s}_n^\varepsilon\le t,\,s_n^\varepsilon>t)$ and $\overline{u}_n^\varepsilon(t)=\mathbb{P}(s_n^\varepsilon\le t,\,\hat{s}_n^\varepsilon>t)$. Since $\overline{u}_n^\varepsilon(t)\le \mathbb{P}(s_n^\varepsilon\le t)$ which is the term of a convergent series (see Proposition \ref{prop:finite-mean}), then for any $\rho>0$ there exists $n_0^\varepsilon\in\mathbb{N}$ such that 
\begin{equation}
\label{eq:rest:series}
\sum_{n>n_0^\varepsilon}\overline{u}_n^\varepsilon(t)<\rho.
\end{equation}
Moreover
\[
\overline{u}_n^\varepsilon(t)=\mathbb{P}(s_n^\varepsilon\le t,\,s_n^\varepsilon+(\hat{s}_n^\varepsilon-s_n^\varepsilon)>t)\le\mathbb{P}(t-\delta<s_n^\varepsilon\le t)+\mathbb{P}(\hat{s}_n^\varepsilon-s_n^\varepsilon>\delta).
\]
The random variables $(s_n^\varepsilon)_{1\le n\le n_0^\varepsilon}$ are absolutely continuous with respect to the Lebesgue measure. Consequently 
\begin{equation}
\label{eq:abs-cont}
\forall\rho>0,\quad \exists \delta_\varepsilon>0 \quad\mbox{such that} \quad \sum_{n=1}^{n_0^\varepsilon}\mathbb{P}(t-\delta_\varepsilon<s_n^\varepsilon\le t)\le \rho.
\end{equation} 
It suffices therefore to deal with the remaining expression: $\sum_{n=1}^{n_0^\varepsilon}\mathbb{P}(\hat{s}_n^\varepsilon-s_n^\varepsilon>\delta_\varepsilon)$. By Step 1 of the proof and by the definition of $(s_n^\varepsilon)$ and $(\hat{s}_n^\varepsilon)$, we have for $n\le n_0^\varepsilon$
\begin{align*}
|s_n^\varepsilon-\hat{s}_n^\varepsilon|&=\varepsilon^2|(\eta^2(x_0^\varepsilon)-\eta^2(\hat{x}_0^\varepsilon))e^{1-A_1}+\ldots+(\eta^2(x_{n-1}^\varepsilon)-\eta^2(\hat{x}_{n-1}^\varepsilon))e^{1-A_n}|\\
&\le \varepsilon^2\, \max_{0\le k\le n-1}|\eta^2(x_k^\varepsilon)-\eta^2(\hat{x}_k^\varepsilon))|(e^{1-A_1}+\ldots +e^{1-A_n})\\
&\le C\, \varepsilon^2 |x-\hat{x}|(e^{1-A_1}+\ldots +e^{1-A_n}),
\end{align*}
with $C=2 ML_{\rm Lip} (1+\varepsilon L_{\rm Lip})^{n_0^\varepsilon}$. Hence
\begin{equation}
\label{eq:triple:rho}
\forall\rho>0,\quad \exists \kappa_\varepsilon>0\quad \mbox{such that}\quad |x-\hat{x}|<\kappa_\varepsilon \  \Rightarrow \ \sum_{n=1}^{n_0^\varepsilon}\mathbb{P}(\hat{s}_n^\varepsilon-s_n^\varepsilon>\delta_\varepsilon)\le \rho.
\end{equation}
Combining \eqref{eq:rest:series}, \eqref{eq:abs-cont} and \eqref{eq:triple:rho}, we obtain
\[
\forall \rho>0,\quad \exists \kappa_\varepsilon>0\quad \mbox{such that}\quad |x-\hat{x}|<\kappa_\varepsilon \  \Rightarrow \ \sum_{n\ge 1}\overline{u}_n^\varepsilon(t)\le 3\rho.
\]
Let us use now similar arguments in order to bound the series associated to  $u_n^\varepsilon(t)$. Using the following upper-bound,
\[
u_n^\varepsilon(t)=\mathbb{P}(s_n^\varepsilon> t,\,s_n^\varepsilon+(\hat{s}_n^\varepsilon-s_n^\varepsilon)\le t)\le\mathbb{P}(t<s_n^\varepsilon\le t+\delta)+\mathbb{P}(\hat{s}_n^\varepsilon-s_n^\varepsilon>\delta),
\]
we deduce that all arguments presented so far and concerning $\overline{u}_n^\varepsilon(t)$ can be used for $u_n^\varepsilon(t)$. Finally \eqref{eq:bound-diff} leads to the continuity of $x\mapsto \psi^\varepsilon(t,x)$:
\[
\forall \rho>0,\quad \exists \kappa_\varepsilon>0\quad \mbox{such that}\quad |x-\hat{x}|<\kappa_\varepsilon \ \Rightarrow \ | \psi^\varepsilon(t,x)-\psi^\varepsilon(t,\hat{x}) |\le 6\rho.
\]
Let us end the proof by focusing our attention on the continuity with respect to the time variable. Since 
\[
\psi^\varepsilon(t,x)=\sum_{n\ge 1}\mathbb{P}(s_n^\varepsilon\le t),
\]
where $s_n^\varepsilon$ is an absolutely continuous random variable, the Lebesgue monotone convergence theorem implies the continuity of $t\mapsto\psi^\varepsilon(t,x)$.
\end{proof}
 We give now an important result concerning the function $\psi^\varepsilon (t,x)=\mathbb{E}[N_t^\varepsilon|x_0^\varepsilon=x]$.
\begin{proposition}
\label{epsilon2-psiepsilon} Under Assumption \ref{assu1} and Assumption \ref{assu2}
, there exist $C_T>0$, $\kappa>0$ and $\varepsilon_0>0$ such that 
\begin{equation}\label{eq:expo-bound}
\varepsilon^2 \psi^\varepsilon (t,x) \le  C_T\,e^{\kappa |x|},\quad \forall t\in[0,T],\ \forall x\in\R,\ \forall\varepsilon\le \varepsilon_0.
\end{equation}
Let us note that the constant $\kappa$ is explicit: it suffices to choose $\kappa=3\sqrt{2}\kappa_2$ where $\kappa_2$ corresponds to the constant introduced in Assumption \ref{assu1}.
\end{proposition}
\begin{proof} 
The proof follows similar ideas as those developed in Proposition \ref{prop:finite-mean}. We introduce here $m_n^\varepsilon:=\min_{0\le k\le n-1}\eta^2(x_k^\varepsilon)$ and recall that
\begin{equation}
\psi^\varepsilon(t,x) = \ds\sum_{n\ge 1} \mathbb{P} (s_n^\varepsilon \le t).
\end{equation}
%Each term is splitted as follows 
%\begin{equation}
%\mathbb{P}  (s_n^\varepsilon \le t) = \mathbb{P}  (s_n^\varepsilon \le t; m_n^\varepsilon \ge \varepsilon^\gamma n^{-2/3} ) + \mathbb{P}  (s_n^\varepsilon \le t; m_n^\varepsilon < \varepsilon^\gamma n^{-2/3}),
%\end{equation}
%where a suitable choice of the exposant $\gamma$ should ensure the required boundedness. We shall discuss this choice in the following.
\noindent We aim to control 
\begin{equation}
\varepsilon^2 \psi^\varepsilon(t,x) =\sum_{n\ge 1}(\varepsilon^2 u_n^\varepsilon(t,x)+\varepsilon ^2 v_n^\varepsilon(t,x) )
\end{equation}
by using similar notations as those presented in Proposition  \ref{prop:finite-mean}, that is 
\begin{align}\label{eq:11}
u_n^\varepsilon(t,x) =  \mathbb{P}  (s_n^\varepsilon \le t; m_n^\varepsilon \ge \varepsilon^\gamma n^{-2/3}|x_0^\varepsilon=x ), \\
v_n^\varepsilon(t,x) =\mathbb{P}  (s_n^\varepsilon \le t; m_n^\varepsilon < \varepsilon^\gamma n^{-2/3}|x_0^\varepsilon=x).
\end{align}
Here a suitable choice of the exponent $\gamma$ should ensure the required boundedness. We shall discuss about this choice in the following.\\
{\bf Step 1.} Consider first the sequence $(u_n^\varepsilon(t,x))_{n\ge 1}$.
Let us introduce $n_\varepsilon^T:=3^{9/2}(2T/e)^{3} \varepsilon^{-2}$ and decompose the series associated to $u_n^\varepsilon$ as follows:
\begin{equation}\label{eq:compl1}
\begin{array}{ll}
\ds\sum_{n\ge 1}\varepsilon^2 u_n^\varepsilon(t,x)&=\varepsilon^2\sum_{n=1}^{n_\varepsilon^T} u_n^\varepsilon(t,x)+\varepsilon^2\sum_{n>n_\varepsilon^T}u_n^\varepsilon(t,x)\\
\ds &\ds\le \varepsilon^2\cdot n_\varepsilon^T+\varepsilon^2\sum_{n>n_\varepsilon^T}u_n^\varepsilon(t,x)\le 3^{9/2}(2T/e)^{3}+\varepsilon^2\sum_{n>n_\varepsilon^T}u_n^\varepsilon(t,x).
\end{array}
\end{equation}
We focus our attention on the indices satisfying $n> n_\varepsilon^T$. Let us set $\gamma=-4/3$. For this particular choice, both the definition \eqref{eq:11} and the definition of $n_\varepsilon^T$ lead to
\begin{equation*}
\begin{array}{ll}
\ds u_n^\varepsilon(t,x)%& =  \mathbb{P}  (s_n^\varepsilon \le t; m_n^\varepsilon \ge \varepsilon^\gamma n^{-2/3} ) \\
& \le \mathbb{P} \left( e^{1-A_1} +\ldots + e^{1-A_n} \le tn^{2/3}\varepsilon ^{-(\gamma+2)}\right)\\[4pt]
\ds & \ds\le \mathbb{P} \left( e^{1-A_1} +\ldots + e^{1-A_n} -n\mathbb{E}[e^{1-A_1}] \le Tn^{2/3}\varepsilon ^{-(\gamma+2)}-n e 3^{-3/2}\right)\\[4pt]
\ds &\ds \le \mathbb{P} \left( e^{1-A_1} +\ldots + e^{1-A_n} -n\mathbb{E}[e^{1-A_1}] \le -Tn^{2/3}\varepsilon^{-(\gamma+2)}\right).
\end{array}
\end{equation*}
Using Hoeffding's inequality permits to obtain the following upper-bound:
\begin{align}\label{eq:compl2}
\varepsilon^2\sum_{n>n_\varepsilon^T}u_n^\varepsilon(t,x)\le \varepsilon^2\sum_{n>n_\varepsilon^T}\exp\left(- \frac{2T^2n^{4/3}\varepsilon^{-2(\gamma+2)}}{ne^2}  \right)\le \varepsilon^2\sum_{n\ge 1}\exp\left(- \frac{2T^2n^{1/3}}{e^2\varepsilon^{4/3}}\right).
\end{align}
Combining \eqref{eq:compl1} and \eqref{eq:compl2}, we obtain for any $\varepsilon\le 1$:
\begin{align*}
\sum_{n\ge 1}\varepsilon^2 u_n^\varepsilon(t,x)\le 3^{9/2}(2T/e)^{3}+\sum_{n\ge 1}\exp\left(- \frac{2T^2\,n^{1/3}}{e^2}\right)=:C_T^0<+\infty.
\end{align*}

Let us just note that the upper-bound does not depend on the space variable $x$.
\par\noindent{\bf Step 2.}  Let us focus now on the second part, that is, the terms $\varepsilon^2 v_n^\varepsilon(t,x)$. Using the properties of $F$ (see Assumption \ref{assu1}), there exist $C>0$ and $\kappa>0$ such that $\eta(z) \ge Ce^{-\kappa |z|}$ (the value of $\kappa$ here corresponds to $\sqrt{2}\kappa_2$ where $\kappa_2$ is the constant appearing in Assumption \ref{assu1})
\begin{equation}
\begin{array}{ll}
\varepsilon ^2 v_n^\varepsilon(t,x) %& =\varepsilon^2 \mathbb{P}  (s_n^\varepsilon \le t; m_n^\varepsilon < \varepsilon^\gamma n^{-2/3})\\
&\le \varepsilon ^2 \mathbb{P} \left(\exists s\le t \mbox{ s.t. } \eta^2(B_s+x)< \varepsilon ^\gamma n^{-2/3} \right)\\
&  \le \varepsilon^2 \mathbb{P} \Big( C^2 \exp\Big( -2\kappa\ds\sup_{s\in[0,t]} |B_s+x| \Big) <\varepsilon^\gamma n^{-2/3} \Big) \\
&  \le \varepsilon^2 \mathbb{P} \Big(   \exp\Big( -\ds\sup_{s\in[0,t]} |B_s+x| \Big) <\Big(C^{-1}\varepsilon^{\frac{\gamma}{2}} n^{-1/3}\Big)^{1/\kappa} \Big),
\end{array}
\end{equation}
for all $(t,x)\in[0,T]\times \R$.
We need here an auxiliary result. 
\begin{lemma}
\label{lemma:R} Let us define the function
\begin{equation}
\mathcal{R}_t(x,z) :=\mathbb{P}\Big( \ds\sup_{s\in[0,t]} |x+B_s|  >\ln(z)\Big).
\end{equation} 
 For any $(x,z)\in\R\times(0,\infty)$ we have $\mathcal{R}_t(x,z) \geq 0$  and $z\mapsto \mathcal{R}_t(x,z)$ is non increasing. Furthermore for any $\delta>0$, there exists $C_T>0$ such that
\begin{equation}
\label{ineg:R}
\ds\int_0^{+\infty} \mathcal{R}_t(x,z^{1/\delta})\, \dint z \le C_T\  e^{\delta |x|},\quad \forall (t,x)\in[0,T]\times \R.
\end{equation}
\end{lemma}
\noindent We postpone the proof of this lemma. We observe therefore
\begin{equation}
v_n^\varepsilon(t,x)  \le \mathcal{R}_t\left(x, \Big(\frac{C n^{1/3}}{\varepsilon^{\gamma/2}}\Big)^{1/\kappa} \right) = \mathbb{P}\left( \ds\sup_{s\in[0,t]} |x+B_s|  >\frac{1}{\kappa}\ln\left( \frac{C n^{1/3}} {\varepsilon^{\gamma/2} }\right)\right).
\end{equation}
We define $n(\varepsilon) =\inf \{ n\geq 0 \mbox{ s.t. } Cn^{1/3} \geq \varepsilon^{\gamma/2}\} $. Then
\begin{equation}
\begin{array}{ll}
\varepsilon^2\ds\sum_{n\ge 1} v_n^\varepsilon(t,x)& = \varepsilon^2\ds\sum_{n=1}^{n(\varepsilon)}  v_n^\varepsilon(t,x)+ \varepsilon^2\ds\sum_{n\ge n(\varepsilon)+1} v_n^\varepsilon(t,x)\\
& \le \varepsilon^2 n(\varepsilon) + \varepsilon^2 \ds\sum_{n\ge n(\varepsilon )+1  } \mathcal{R}_t\left(x, \Big(\ds\frac{Cn^{1/3}}{\varepsilon ^{\gamma/2}}\Big)^{1/\kappa}\right)\\
& \le  \varepsilon^2 n(\varepsilon) + \varepsilon^2 \ds\int_{n(\varepsilon) } ^{+\infty} \mathcal{R}_t\left(x, \Big(\ds\frac{Cz^{1/3}}{\varepsilon ^{\gamma/2}}\Big)^{1/\kappa}\right)\dint z.
\end{array}
\end{equation} 
In the last expression only the term under the integral depends on $x$. We perform the change of variable in this term of the form $y=\varepsilon^2 z$ and obtain:
\begin{equation}
\begin{array}{ll}
\varepsilon^2\ds\sum_{n\ge 1} v_n^\varepsilon(t,x)& \le  \varepsilon^2 n(\varepsilon) + \ds\int_{\varepsilon^2n(\varepsilon) } ^{+\infty} \mathcal{R}_t\left(x, \Big(\ds\frac{Cy^{1/3}}{\varepsilon ^{2/3+\gamma/2}}\Big)^{1/\kappa}\right)\dint y\\
& \le  \varepsilon^2 n(\varepsilon) + \ds\int_0^{+\infty} \mathcal{R}_t(x,C^{1/\kappa}y^{1/3\kappa})\dint y \le  \varepsilon^2 n(\varepsilon) + C_T e^{3\kappa |x|},
\end{array}
\end{equation} 
by using the particular value $\gamma = -4/3$ and Lemma \ref{lemma:R}. In order to conclude we need to control $\varepsilon^2 n(\varepsilon)$.  By the definition of $n(\varepsilon)$ we have
\begin{equation}
 n(\varepsilon) = \ds\inf\left\{ n\ge 0 \mbox{ s.t. } n\ge \ds\frac{1}{\varepsilon^2C ^3}\right\} \le \frac{1}{\varepsilon^2 C^3} +1,\quad \mbox{for}\ \varepsilon\le 1.
\end{equation} 
This allows us to conclude that
$\varepsilon ^2 n(\varepsilon)  \leq \frac{1}{C^3}+1$ for $\varepsilon\le 1$. Combining the two steps of the proof leads to the announced upper-bound \eqref{eq:expo-bound}.
\end{proof}
\begin{proof}[Proof of Lemma \ref{lemma:R}]
The proof of the first two properties is obvious by using the definition of $\mathcal{R}_t$. Let us show that 
\eqref{ineg:R} is true.  By using the reflection principle of the Brownian  motion we can evaluate
\begin{equation}
\label{intermediar:R}
\begin{array}{ll}
\ds\int_0^{+\infty} \mathcal{R}_t(x,z^{1/\delta}) \dint z & =\ds\int_0^{e^{\delta(1+|x|)}} \mathcal{R}_t(x,z^{1/\delta})\dint z + \ds\int_{e^{\delta(1+|x| )}} ^{+\infty} \mathcal{R}_t(x,z^{1/\delta})\dint z\\
& \le e^{\delta(1+|x|)} + 4\ds\int _{e^{\delta(1+|x|)}} ^{+\infty} \mathbb{P} \Big( G >\ds\frac{1}{\sqrt{t}} \Big(\ds\frac{1}{\delta} \ln(z) -|x|\Big) \Big) \dint z\\
& \le e^{\delta(1+|x|)} + 4\ds\int _{e^{\delta(1+|x|)}} ^{+\infty} \mathbb{P} \Big( G >\ds\frac{1}{\sqrt{T}} \Big(\ds\frac{1}{\delta} \ln(z) -|x|\Big) \Big) \dint z, 
\end{array}
\end{equation}
where $G$ denotes a standard normal random variable ${\cal {N}} (0,1)$. We used the fact that for any $u>0 $ we have $\mathbb{P} (G\geq u) \le \frac{1}{u\sqrt{2\pi}} e^{-\frac{u^2}{2}}$. Hence, for $z\ge e^{\delta(1+|x|)}$,
\begin{align*}
\mathbb{P} \Big( G >\ds\frac{1}{\sqrt{T}} \Big(\ds\frac{1}{\delta} \ln(z) -|x|\Big) \Big)&\le  \frac{\delta\sqrt{T}}{(\ln(z)-\delta|x|)\sqrt{2\pi}}\, \exp\Big\{-\frac{1}{2T}\Big(\frac{1}{\delta}\,\ln(z)-|x|\Big)^2\Big\}\\
&\le \frac{\sqrt{T}}{\sqrt{2\pi}}\, \exp\Big\{-\frac{1}{2T}\Big(\frac{1}{\delta}\,\ln(z)-|x|\Big)^2\Big\}.
\end{align*}
We deduce
\begin{align*}
\ds\int_0^{+\infty} \mathcal{R}_t(x,z^{1/\delta}) \dint z &\le e^{\delta(1+|x|)}+ \frac{4\sqrt{T}}{\sqrt{2\pi}}\ds\int _{e^{\delta(1+|x|)}} ^{+\infty}  \, \exp\Big\{-\frac{1}{2T}\Big(\frac{1}{\delta}\,\ln(z)-|x|\Big)^2\Big\}\dint z.
\end{align*}
By doing the change of variable $z= e^{\delta(1+|x|)} y$, we have
\begin{align*}
\ds\int_0^{+\infty} \mathcal{R}_t(x,z^{1/\delta}) \dint z & \le e^{\delta(1+|x|)}\Big[ 1+ \ds\frac{4\sqrt{T}}{\sqrt{2\pi }} \ds\int_1^{+\infty} e^{\ds-\frac{1}{2T}( \frac{1}{\delta} (\delta(1+|x|) +\ln y) -|x|)^2}\dint y \Big]\\
& \le   e^{\delta(1+|x|)}\Big[ 1+\ds\frac{4\sqrt{T}}{\sqrt{2\pi }} \ds\int_1^{+\infty} e^{-\ds\frac{1}{2T}(1+\frac{1}{\delta}\ln(y))^2} \dint y \Big]=:C_T\,e^{\delta |x|}.
\end{align*}
The upper-bound holds for any $(t,x)\in[0,T]\times\R$ as announced.
\end{proof}
}
%%AICI
Since Proposition \ref{prop:finite-mean} and Proposition \ref{prop:conti} point out different preliminary properties of the average number of steps needed by the Brownian skeleton to cover the time interval $[0,T]$, the study of the $\varepsilon$-strong approximation of both the linear and growth diffusions can be achieved. 
\begin{thm}\label{thm:efficient:diff}
Let $(X_t)_{0\le t\le T}$ be a solution of the stochastic differential equation  \eqref{eq:eds} satisfying both Assumptions \ref{assu1} and \ref{assu2}. Let $(y^\varepsilon_t)_{0\le t\le T}$ be the $\varepsilon$-strong approximation of $(X_t)_{0\le t\le T}$ given by \eqref{eq:approx:diff} and $\hat{N}^\varepsilon_T$ the random number of points needed to build this approximation. Then, there exist $\mu>0$  such that
\begin{equation}
\label{eq:thm:diff:efficient}
\lim_{\varepsilon\to 0}\varepsilon^2\,\mathbb{E}[\hat{N}^\varepsilon_T]=\mu\,\mathbb{E}\left[ \ds\int_0^{\rho(T)}  \frac{1}{\eta^2(x+B_s)}\dint s\right],\quad \forall (T,x)\in\R_+\times\R
\end{equation}
where $(B_t)_{t\ge 0}$ stands for a standard one-dimensional Brownian motion.
\end{thm}
%\begin{Com} Dans notre étude, le mouvement Brownien est parfois "standard" et commence en $0$, parfois il commence en $x$. Il faudrait mener une petite réflexion sur l'uniformisation des concepts utilisés.
%\end{Com}
\begin{rem} \label{rem1}\begin{enumerate}
\item The constant appearing in the statement is explicitly known. Let us introduce $M$ the cumulative distribution function associated to the random variable $e^{1-A}$ with $A\sim{\Gamma}(3/2,2)$. We denote by $M(f)=\int_0^\infty f(s)\,\dint M(s)$, for any nonnegative function $f$. Then
\begin{equation}
\label{eq:ajust:const}
\mu=\frac{1}{M(\phi_1^2)}\quad \mbox{and}\quad M(\phi_1^2)=M({\rm Id})=e\,3^{-3/2}\approx 0.5231336.
\end{equation}
\item Let us note that the link between the function $\eta$, defining the approximation scheme, and the function $F$ introduced in Assumption \ref{assu1}, permits to write
\begin{align*}
\lim_{\varepsilon\to 0}\varepsilon^2\,\mathbb{E}[\hat{N}^\varepsilon_T]&=(e\kappa_{\rm min}^{-1}+1)^2\,\mu\,\mathbb{E}\left[ \ds\int_0^{\rho(T)}  F^2(2(x+B_s)^2+1)\dint s\right]\\
&=(e\kappa_{\rm min}^{-1}+1)^2\,\mu\,\mathbb{E}\left[ \ds\int_\R  F^2(2(x+y)^2+1) L_{\rho(T)}  (y) \dint y\right].
\end{align*}
The last equality is just an immediate application of the occupation time formula (see, for instance, Corollary 1.6 page 209 in \cite{Revuz-Yor}), $L_t(y)$ standing for the local time of the standard Brownian motion. A proof of Theorem \ref{thm:efficient:diff} based on the local times of the Brownian motion and therefore on a precise description of the Brownian paths could be investigated, we prefer here to propose a proof involving a renewal property of the average number of points in the numerical scheme.
\end{enumerate}
\end{rem}
\begin{proof}[Proof of Theorem \ref{thm:efficient:diff}] We start to mention that the notation of the constants is generic through this proof: $C:=C_\theta$ or $\kappa:=\kappa_\theta$ if the constants depend on a parameter $\theta$.\\ The proof of the theorem is based on the study of a particular renewal inequality. The material is organized as follows: on one hand we shall prove that 
$u^\varepsilon(t,x):=\varepsilon^2\mathbb{E}[N_t^\varepsilon|x_0^\varepsilon=x]$, $N^\varepsilon_t$ being defined in \eqref{eq:def:count}, satisfies a renewal equation. On the other hand, we describe $U(t,x)$ defined by 
\begin{equation}
\label{eq:def:Utx}
U(t,x):=\mu\mathbb{E}\left[ \ds\int_0^{ t}  \frac{1}{\eta^2(x+B_s)}\dint s\right]=\frac{(e\kappa_{\rm min}^{-1}+1)^2}{M(\phi_1^2)}\,\mathbb{E}\left[ \ds\int_0^{t}  F^2(2(x+B_s)^2+1)\dint s\right],
\end{equation}
where $\mu$ corresponds to the constant introduced in the statement of the theorem and described in Remark \ref{rem1}. Then we observe that the difference:
\begin{equation}\label{eq:def:diff}
D^\varepsilon(t,x):=u^\varepsilon(t,x)-U(t,x)
\end{equation}
satisfies a renewal inequality which leads to $\lim_{\varepsilon\to 0}D^\varepsilon(t,x)=0$.\\[5pt]
{\bf Step 1. Renewal equation satisfied by \mathversion{bold}$\psi^\varepsilon(t,x)=\mathbb{E}[N^\varepsilon_t|x_0^\varepsilon=x]$\mathversion{normal}}.
Let us note that $\psi^\varepsilon(t,x)$ satisfies the following renewal equation:
\begin{equation}
\label{eq:renew}
\psi^\varepsilon(t,x)=M(t/(\varepsilon^2\eta^2(x)))+\sum_{i=\pm 1}\frac{1}{2}\int_0^{t/(\varepsilon^2\eta^2(x))}\psi^\varepsilon\Big(t-s\varepsilon^2\eta^2(x),\ x+i\varepsilon\eta(x)\phi_1(s)\Big)\,\dint M(s)
\end{equation}
where $M$ corresponds to the cumulative distribution function associated to the random variable $e^{1-A}$ with $A\sim{\Gamma}(3/2,2)$.
Indeed, we focus our attention to $U_1^\varepsilon$ the first positive abscissa of the Brownian paths skeleton $({\rm BS})_\eta$. 

We can observe two possibilities:
\begin{itemize}
\item Either $U_1^\varepsilon>t$ and consequently $N_t^\varepsilon=0$.
\item Either $U_1^\varepsilon=s\varepsilon^2\eta^2(x)\le t$. The Markov property implies the following identity in distribution: for any non negative measurable function $\mathcal{H}$, we have
\[
\mathbb{E}[\mathcal{H}(N_t^\varepsilon)|x_0^\varepsilon,\, U_1^\varepsilon]=\mathbb{E}[\mathcal{H}_0(t-s\varepsilon^2\eta^2(x), x_1^\varepsilon)|x_0^\varepsilon,\, U_1^\varepsilon]\ \mbox{with}\ \mathcal{H}_0(t,y)=\mathbb{E}[\mathcal{H}(1+N^\varepsilon_{t})|x_0^\varepsilon=y].
\]
We just note that the function $\mathcal{H}_0(t,y)$ associated with  $\mathcal{H}(y)=y1_{y\ge 0}$ corresponds to $1+\psi^\varepsilon(t,y)$.
\end{itemize}
 Hence
\begin{align*}
\psi^\varepsilon(t,x)
%&=\int_0^{t/(\varepsilon^2\eta^2(x))}\mathbb{E}\left[ \mathbb{E}\left[1+N^\varepsilon_{t-s\varepsilon^2\eta^2(x)}\Big| x_1^\varepsilon\right]\,\Big|x_0^\varepsilon=x\right]\,\dint M(s)\\
&=M(t/(\varepsilon^2\eta^2(x))+\int_0^{t/(\varepsilon^2\eta^2(x))}\mathbb{E}\left[\psi^\varepsilon(t-s\varepsilon^2\eta^2(x),x_1^\varepsilon)\,\Big|x_0^\varepsilon=x\right]\,\dint M(s).
\end{align*}
Let us now introduce $u^\varepsilon(t,x):=\varepsilon^2 \psi^\varepsilon(t,x)$ and define for any nonnegative function $h$:
\begin{equation}
\label{eq:def:P}
\mathcal{M}_\varepsilon h(t,x)=\sum_{i=\pm 1}\frac{1}{2}\int_0^{t/(\varepsilon^2\eta^2(x))}h\Big(t-s\varepsilon^2\eta^2(x),\ x+i\varepsilon\eta(x)\phi_1(s)\Big)\,\dint M(s).
\end{equation}
Then the following renewal equation holds
\begin{align}\label{eq:renounew}
u^\varepsilon(t,x)&=\varepsilon^2M(t/(\varepsilon^2\eta^2(x)))+\mathcal{M}_\varepsilon u^\varepsilon(t,x),\quad \forall (t,x)\in\mathbb{R}_+\times\mathbb{R}.
%&=\varepsilon^2M(t/\eta^2_\varepsilon(x))+u^\varepsilon(t,x)M(t/\eta^2_\varepsilon(x))-\frac{\partial u}{\partial t}(t,x)\eta_\varepsilon^2(x)\int_0^{t/\eta_\varepsilon^2(x)}s\,\dint M(s)\\
%&\quad +\frac{1}{2}\frac{\partial ^2 u}{\partial x^2}(t,x)\eta_\varepsilon^2(x)\int_0^{t/\eta_\varepsilon^2(x)}\phi_1^2(s)\,\dint M(s)+o_I(\varepsilon^2),\quad \forall x\in I.
\end{align}
\mathversion{bold}
\noindent {\bf Step 2. Description of the function $U(t,x)$ introduced in \eqref{eq:def:Utx}.}
\mathversion{normal}
%Let us define
%\[
%U(t,x):=\frac{(e\kappa_{\rm min}^{-1}+1)^2}{M(\phi_2)}\,\ \mathbb{E}_x\left[ \ds\int_0^{tM(\phi_1^2)/M({\rm Id})}  F^2(|B_s|+1)\dint s\right],
%\]
%where $(B_t)$ stands for a one-dimensional Brownian motion. 
Due to Assumption \ref{assu1}, $F$ is assumed to be a $\mathcal{C}^2(\mathbb{R})$-continuous function and $F$, $F'$ and $F''$ have at most exponential growth. The dominated convergence theorem permits therefore to obtain that $x\mapsto U(t,x)$ is also a $\mathcal{C}^2(\mathbb{R}, \mathbb{R})$-continuous function. Moreover, combining It\^o's formula and Lebesgue's theorem lead to the regularity with respect to both variables $t$ and $x$: $U$ is $\mathcal{C}^2(\mathbb{R}_+\times\mathbb{R},\mathbb{R})$-continuous. Since $U$ is regular and has at most exponential growth, it corresponds to the probabilistic representation (see for example Karatzas and Shreve \cite{karatzas-shreve-1991}, p. 270, Corollary 4.5) of the unique solution:
\begin{equation}
\label{PDE}
M({\rm Id})\frac{\partial U}{\partial t}(t,x)= \frac{M(\phi_1^2)}{2}\frac{\partial ^2 U}{\partial x^2}(t,x)+(e\kappa_{\rm min}^{-1}+1)^2F^2(2x^2+1),\quad U(0,x)=0.
\end{equation}
We just recall that $M({\rm Id})=M(\phi_1^2)$ (see Remark \ref{rem1}).\\
Let $R>0$. Using the Taylor expansion in order to compute the operator defined in \eqref{eq:def:P}, we obtain
\begin{align*}
\mathcal{M}_\varepsilon U(t,x) =&U(t,x) M(t/(\varepsilon^2\eta^2(x)))-\varepsilon^2\eta^2(x)M({\rm Id})\frac{\partial U}{\partial t}(t,x)\\
&+\frac{\varepsilon^2}{2}\eta^2(x)M(\phi_1^2)\frac{\partial^2U}{\partial x^2}(t,x)+\varepsilon^2o_{R}(1),
\end{align*}
where $o_{R}(1)$ tends uniformly towards $0$ on $[0,T]\times [-R,R]$ as $\varepsilon \to 0$ (the uniformity of the reminder term can be observed with classical computations, let us just note that $o_{R}(1)$ is a generic notation in the sequel). The equation \eqref{PDE} and the particular link between both functions $F$ and $\eta$ imply
\begin{align}\label{eq:operU}
\mathcal{M}_\varepsilon U(t,x)&=U(t,x) M(t/(\varepsilon^2\eta^2(x)))-\varepsilon^2+\varepsilon^2o_{R}(1), \quad \forall(t,x)\in [0,T]\times [-R,R].
\end{align}
\mathversion{bold}
\noindent {\bf Step 3. Study of the difference $D^\varepsilon(t,x)$ introduced in \eqref{eq:def:diff}.}
\mathversion{normal} Since both $u^\varepsilon$ and $U$ are continuous functions satisfying an exponential bound (immediate consequence of the regularity and growth property of $F$ for $U$ and statement of Proposition \ref{epsilon2-psiepsilon} for $u^\varepsilon$), so is $D^\varepsilon$. Hence, there exists $C>0$ and $\kappa>0$ such that 
\begin{align}
\label{eq:bound-D}
|D^\varepsilon(t,x)|\le Ce^{\kappa |x|},\quad \forall (t,x)\in[0,T]\times\R.
\end{align}
Moreover combining \eqref{eq:operU} and \eqref{eq:renounew} implies
\begin{align}\label{eq:inter}
D^\varepsilon(t,x)=(U(t,x)+\varepsilon^2)(M(t/(\varepsilon^2\eta^2(x)))-1)+\mathcal{M}_\varepsilon D^\varepsilon(t,x)+\varepsilon^2 o_{R}(1).
\end{align}
The support of the distribution associated to $M$ is compact. Moreover $\eta$ defined in \eqref{def:eta1} is upper-bounded. Consequently there exists $\rho>0$ (independent of $x$ and $\varepsilon$) such that $M(t/(\varepsilon^2\eta^2(x)))=1$ for all $t\ge \rho\varepsilon^2$ and $x\in\R$. For small values of $t$, that is $t\le \rho\varepsilon^2$, it suffices to use the regularity of $U$ with respect to that variable in order to get a constant $C_R>0$ such that $|U(t,x)|\le C_R \varepsilon^2$ for all $x\in[-R,R]$. To sum up the observations for any value of $t$: there exists 
$C_{R}>0$ such that
\begin{align}
\label{eq:ineq}
|D^\varepsilon(t,x)|\le \mathcal{M}_\varepsilon |D^\varepsilon|(t,x)+(C_{R}+1)\,\varepsilon^21_{\{t\le \rho\varepsilon^2\}}+\varepsilon^2 o_{R}(1),\quad \forall(t,x)\in[0,T]\times[-R,R]. 
\end{align}
\mathversion{bold}
\noindent {\bf Step 4. Asymptotic behaviour of $D^\varepsilon(t,x)$.}
\mathversion{normal}
It is possible to link the operator $\mathcal{M}_\varepsilon$ to the approximation scheme of the Brownian motion: the Brownian skeleton $({\rm BS})_\eta$. We recall that $s_n^\varepsilon=\sum_{k=1}^n U_k^\varepsilon$ and that $(s_n^\varepsilon,x_n^\varepsilon)_{n\ge 0}$ is a skeleton of the Brownian paths: the sequence (Markov chain) has the same distribution than points belonging to a Brownian trajectory. It represents the successive exit times and positions of small $\phi_\varepsilon$-domains also called heat-balls, the radius of any heat-ball being upper-bounded by $\varepsilon\eta(0)$. We observe that 
\[
\mathcal{M}_\varepsilon(h)(t,x)=\mathbb{E}[h(t-U^\varepsilon_1,x^\varepsilon_1)1_{\{U^\varepsilon_1\le t\}}|x_0^\varepsilon=x],\quad\mbox{for any nonnegative function }h.
\]
Consequently, for any $(t,x)\in[0,T]\times[-R,R]$, \eqref{eq:ineq} becomes
\begin{align}
\label{eq:ineq1}
|D^\varepsilon(t,x)|\le  \mathbb{E}[|D^\varepsilon(t-U^\varepsilon_1,x^\varepsilon_1)|1_{\{U^\varepsilon_1\le t\}}|x_0^\varepsilon=x]+C_{R}\,\varepsilon^21_{\{t\le \rho\varepsilon^2\}}+\varepsilon^2 o_{R}(1). 
\end{align}
Since the sequence $(s_n^\varepsilon,x_n^\varepsilon)_{n\ge 0}$ is a Markov chain, the aim is to iterate the upper-bound a large number of times. In order to achieve such a procedure, we need to ensure that $x_1^\varepsilon,\ldots, x_n^\varepsilon$ belong to the interval $[-R,R]$. We introduce 
\[
\tau_{R,\varepsilon}=\inf\{n\ge 0:\ x_n^\varepsilon\notin [-R,R]\}.
\]
The $\phi_\varepsilon$-domains associated to the Brownian approximation are bounded (their radius is less than $\varepsilon\eta(0)$), we therefore obtain that $|x_{\tau_{R,\varepsilon}}|\le R+\varepsilon\eta(0)$ and \eqref{eq:bound-D} implies the existence of $C>0$ and $\kappa>0$ such that $|D^\varepsilon(t,x_{\tau_{R,\varepsilon}})|\le Ce^{\kappa R}$, for any $t\le T$ and $\varepsilon\le 1$. Let us note that for notational convenience we use $\mathbb{P}_x$ (resp. $\mathbb{E}_x$) for the conditional probability (resp. expectation) with respect to the event $x_0^\varepsilon=x$.  Hence \eqref{eq:ineq} gives
\begin{align*}
|D^\varepsilon(t,x)|&\le  \mathbb{E}_x[|D^\varepsilon(t-s^\varepsilon_1,x^\varepsilon_1)|1_{\{s^\varepsilon_1\le t\,;\ |x_1^\varepsilon|\le R\}}] +\mathbb{E}_x[|D^\varepsilon(t-s^\varepsilon_1,x^\varepsilon_1)|1_{\{s^\varepsilon_{\tau_{R,\varepsilon}}\le t\,;\ \tau_{R,\varepsilon}=1\}}]\\
&\quad +C_{R}\,\varepsilon^21_{\{t\le \rho\varepsilon^2\}}+\varepsilon^2 o_{R}(1)\\
&\le  \mathbb{E}_x[|D^\varepsilon(t-s^\varepsilon_1,x^\varepsilon_1)|1_{\{s^\varepsilon_1\le t\,;\ |x_1^\varepsilon|\le R\}}]+Ce^{\kappa R}\,\mathbb{P}_x(s^\varepsilon_{\tau_{R,\varepsilon}}\le t;\ \tau_{R,\varepsilon}=1)\\
&\quad+C_{R}\,\varepsilon^21_{\{t\le \rho\varepsilon^2\}}+\varepsilon^2 o_{R}(1).
\end{align*}
In order to simply the notations when iterating the procedure, we introduce the following events:
\begin{align*}
%\label{eq:def:events}
\mathcal{J}_{R,\varepsilon}^n:=\{ s^\varepsilon_n\le t\}\cap\{ \tau_{R,\varepsilon}>n\}.
\end{align*}
By iterating the upper-bound, we obtain
\begin{align}
\label{eq:ensem}
|D^\varepsilon(t,x)|&\le  \mathbb{E}_x[|D^\varepsilon(t-s^\varepsilon_2,x^\varepsilon_2)|1_{\mathcal{J}_{R,\varepsilon}^2}]+C_R\,\varepsilon^2\,\mathbb{P}_x(t-\rho\varepsilon^2\le
s_1^\varepsilon\le t\,;\,\mathcal{J}_{R,\varepsilon}^2 )\nonumber\\
&\quad +C\,e^{\kappa R}\,\mathbb{P}_x(s^\varepsilon_{\tau_{R,\varepsilon}}\le t;\ \tau_{R,\varepsilon}\le 2)+C_{R}\,\varepsilon^21_{\{t\le \rho\varepsilon^2\}}+2\,\varepsilon^2 o_{R}(1)\nonumber\\
&\le \mathbb{E}_x[|D^\varepsilon(t-s^\varepsilon_n,x^\varepsilon_n)|1_{\mathcal{J}_{R,\varepsilon}^n}]+C_R\,\varepsilon^2\sum_{k\ge 1}\mathbb{P}_x(t-\rho\varepsilon^2\le
s_k^\varepsilon\le t\,;\mathcal{J}_{R,\varepsilon}^k)\nonumber\\
&\quad +C\,e^{\kappa R}\,\mathbb{P}_x(s^\varepsilon_{\tau_{R,\varepsilon}}\le t;\ \tau_{R,\varepsilon}\le n)+C_{R}\,\varepsilon^21_{\{t\le \rho\varepsilon^2\}}+n\,\varepsilon^2 o_{R}(1)\nonumber\\
&\le C_R\Big(\mathcal{A}_1(R,\varepsilon,n)+\mathcal{A}_2(R,\varepsilon)+\mathcal{A}_3(R)+\mathcal{A}_4(R,\varepsilon,n)\Big),
\end{align}
with
\begin{eqnarray*}
\begin{array}{ll}
\ds\mathcal{A}_1(R,\varepsilon,n) :=\,\mathbb{P}_x(\mathcal{J}_{R,\varepsilon}^n),&\ds\mathcal{A}_2(R,\varepsilon)=\varepsilon^2\sum_{k\ge 1}\mathbb{P}_x(t-\rho\varepsilon^2\le
s_k^\varepsilon\le t\,;\mathcal{J}_{R,\varepsilon}^k),\\[5pt]
\ds\mathcal{A}_3(R)=e^{\kappa R}\,\mathbb{P}(\exists s\le  t:\, x+B_s\notin [-R,R]),&\mathcal{A}_4(R,\varepsilon,n)=\varepsilon^21_{\{t\le \rho\varepsilon^2\}}+n\,\varepsilon^2 o_{R}(1).
\end{array}
\end{eqnarray*}
We shall now describe precisely the bound of each of these terms. The crucial idea is to first fix $R$ sufficiently large and then to choose $n=\xi\lfloor 1/\varepsilon^2\rfloor$ for $\xi$ large enough and depending on $R$.
Let $\delta>0$. We shall prove that there exists $\varepsilon_0$ such that $|D^\varepsilon(t,x)|\le \delta$ for $\varepsilon\le \varepsilon_0$.
\begin{enumerate}
\item Due to the reflection principle of the Brownian motion, there exists $R$ large enough such that
\begin{equation}\label{eq:111}
\mathcal{A}_3(R)\le 4e^{\kappa R}\,\mathbb{P}(B_t\ge R-|x|)=4e^{\kappa R}\,\mathbb{P}\Big(G\ge \frac{R-|x|}{\sqrt{t}}\Big)\le \delta/4,
\end{equation}
where $G$ is a standard Gaussian variate. From now on, $R$ is fixed s.t. \eqref{eq:111} is satisfied.
\item Let us consider the term $\mathcal{A}_2$. We introduce the particular choice $n=\xi\lfloor 1/\varepsilon^2\rfloor$ with $\xi\in\mathbb{N}$. By the definition of the Brownian skeleton, $s_n^\varepsilon\le t$ corresponds to
\[
\sum_{k=1}^{\xi\lfloor 1/\varepsilon^2\rfloor}\varepsilon^2\eta^2(x_{k-1}^\varepsilon)e^{1-A_k}\le t,
\]
where $(A_k)_{k\ge 1}$ is a sequence of i.i.d random variables. Since $\eta$ is an even function and decreases on $\R_+$, we observe :
\[
\mathcal{J}_{R,\varepsilon}^n\subset \left\{\frac{\varepsilon^2}{\xi}\sum_{k=1}^{\xi\lfloor 1/\varepsilon^2\rfloor}e^{1-A_k}\le \frac{t}{\xi\eta^2(R)},\right\},\quad \forall\xi\in\mathbb{N}.
\]
By the law of large numbers, the left hand side of the inequality converges towards $\mathbb{E}[e^{1-A_1}]$ as $\varepsilon\to 0$. Hence, as soon as $\xi>t/(\eta^2(R)\mathbb{E}[e^{1-A_1}])$, there exists $\varepsilon_1>0$ such that $\mathcal{A}_1(R,\varepsilon,n)\le \delta/4$ for $\varepsilon\le \varepsilon_1$ and $n=\xi\lfloor 1/\varepsilon^2\rfloor$. 
\item Let us now deal with $\mathcal{A}_4$. The parameters $R$ and $\xi$ have already been fixed and $n=\xi\lfloor 1/\varepsilon^2\rfloor$. It is therefore obvious that there exists a constant $\varepsilon_2>0$ such that $\mathcal{A}_4(R,\varepsilon,n)\le \delta/4$ for $\varepsilon\le \varepsilon_2$.
\item Finally we focus our attention on the last term $\mathcal{A}_2$ ($R$ being fixed). We introduce the notation $\chi(A,B)=1_{A\cap B}$ and the stopping time
\[
\mathfrak{z}=\inf\{n\ge 0:\ s_n^\varepsilon\ge t-\rho\varepsilon^2\}.
\] 
Then
\begin{align*}
\varepsilon^{-2}\mathcal{A}_2(R,\varepsilon)&=\mathbb{E}\left[ \sum_{k\ge 1} \chi(t-\rho\varepsilon^2\le s_k^\varepsilon\le t, k<\tau_{R,\varepsilon})\right]=\mathbb{E}\left[ \sum_{k\ge \mathfrak{z}} \chi(s_k^\varepsilon\le t, k<\tau_{R,\varepsilon})\right]\\
&\le 1+\mathbb{E}\left[ \sum_{k\ge 1} \chi(U^\varepsilon_{\mathfrak{z}+1}+\ldots+U^\varepsilon_{\mathfrak{z}+k}  \le t-s_\mathfrak{z}^\varepsilon, \mathfrak{z}+k<\tau_{R,\varepsilon})\right].
\end{align*}
By definition $U^\varepsilon_n=\varepsilon^2\eta^2(x_{n-1}^\varepsilon)e^{1-A_n}\ge \varepsilon^2\eta^2(R)e^{1-A_n}$ for any $n<\tau_{R,\varepsilon}$, since $\eta$ is decreasing on $\mathbb{R}_+$ and corresponds to an even function. Moreover the definition of $\mathfrak{z}$ implies $t-s_\mathfrak{z}^\varepsilon\le \rho\varepsilon^2$. We deduce that
\begin{align*}
\varepsilon^{-2}\mathcal{A}_2(R,\varepsilon)&\le 1+\mathbb{E}\left[ \sum_{k\ge 1} \chi\Big(e^{1-A_{\mathfrak{z}+1}}+\ldots+e^{1-A_{\mathfrak{z}+k}}  \le\frac{\rho}{\eta^2(R)}, \mathfrak{z}+k<\tau_{R,\varepsilon}\Big)\right].
\end{align*}
Since $(e^{1-A_n})_{n\ge 1}$ is a sequence of i.i.d random variables, we can define the associate renewal process $(\overline{N}_t)_{t\ge 0}$ already introduced in the proof of  Proposition \ref{prop:ren}. We obtain
\begin{align*}
\varepsilon^{-2}\mathcal{A}_2(R,\varepsilon)&\le 1+\mathbb{E}\left[ \sum_{k\ge 1} \chi\Big(e^{1-A_1}+\ldots+e^{1-A_k}  \le\frac{\rho}{\eta^2(R)}\Big)\right]\le 1+\mathbb{E}[\overline{N}_{\frac{\rho}{\eta^2(R)}}]<\infty.
\end{align*}
In other words, there exists $\varepsilon_3>0$ s.t. $\mathcal{A}_2(R,\varepsilon)\le \delta/4$ for any $\varepsilon\le \varepsilon_3$.
\end{enumerate}
Let us combine the asymptotic analysis of each term in \eqref{eq:ensem}. Then, for any $\delta>0$, we define $\varepsilon_0:=\min(\varepsilon_1,\varepsilon_2,\varepsilon_3)$ which insures the announced statement: $|D^\varepsilon(t,x)|\le \delta$ for any $\varepsilon\le\varepsilon_0$. 
\end{proof}

\section{Numerical application}
Let us focus our attention on particular examples of $L$-class diffusion processes. We recall that these diffusion processes are characterized by their drift term $a(t)x+b(t)$ and their diffusion coefficient $\overline{\sigma}(t)$. In many situations, both the particular function $f(t,x)$ and the time scale $\rho(t)$ which permit to write the diffusion process as a function of the time-changed Brownian motion $X_t=f(t,x_0+B_{\rho(t)})$ have an explicit formula. We propose two particular cases already introduced in exit problem studies \cite{herrmann2020approximation}.

For each one of these cases we first illustrate some of the results obtained in the theoretical part. Secondly we compare our approach with classical schemes like the Euler scheme. Even if this comparaison is quite difficult as our method looks for a control on the path with an $\varepsilon$ approximation while the classical methods do not follow this objective, we construct a rough comparaison that we explain later on.  \\[5pt]
{\bf Example 1 (periodic functions).} We set:
\begin{equation}
\label{def:ex:1}
a(t)=\frac{\cos(t)}{2+\sin(t)},\quad b(t)=\cos(t),\quad\mbox{and}\quad  \overline{\sigma}(t)=2+\sin(t).
\end{equation}
Then the three basic components of the $\varepsilon$-strong approximation (see Theorem \ref{thm:General}) are given by  $\rho(t)=4t$,
\[
f(t,x)=(2+\sin(x))\Big( \frac{x}{2}+\ln\Big(1+\frac{\sin(t)}{2}\Big) \Big)\quad\mbox{and}\quad F(x)=3+\frac{\sqrt{|x|}}{2}.
\]
We observe that the simulation of a $\varepsilon$-strong approximation of the diffusion paths $(X_t,\ t\in[0,1])$ requires a random number of $\phi_\varepsilon$-domains illustrated by the histogram in Figure \ref{Fig}. 

\begin{figure}[h]
\centerline{\includegraphics[width=8cm]{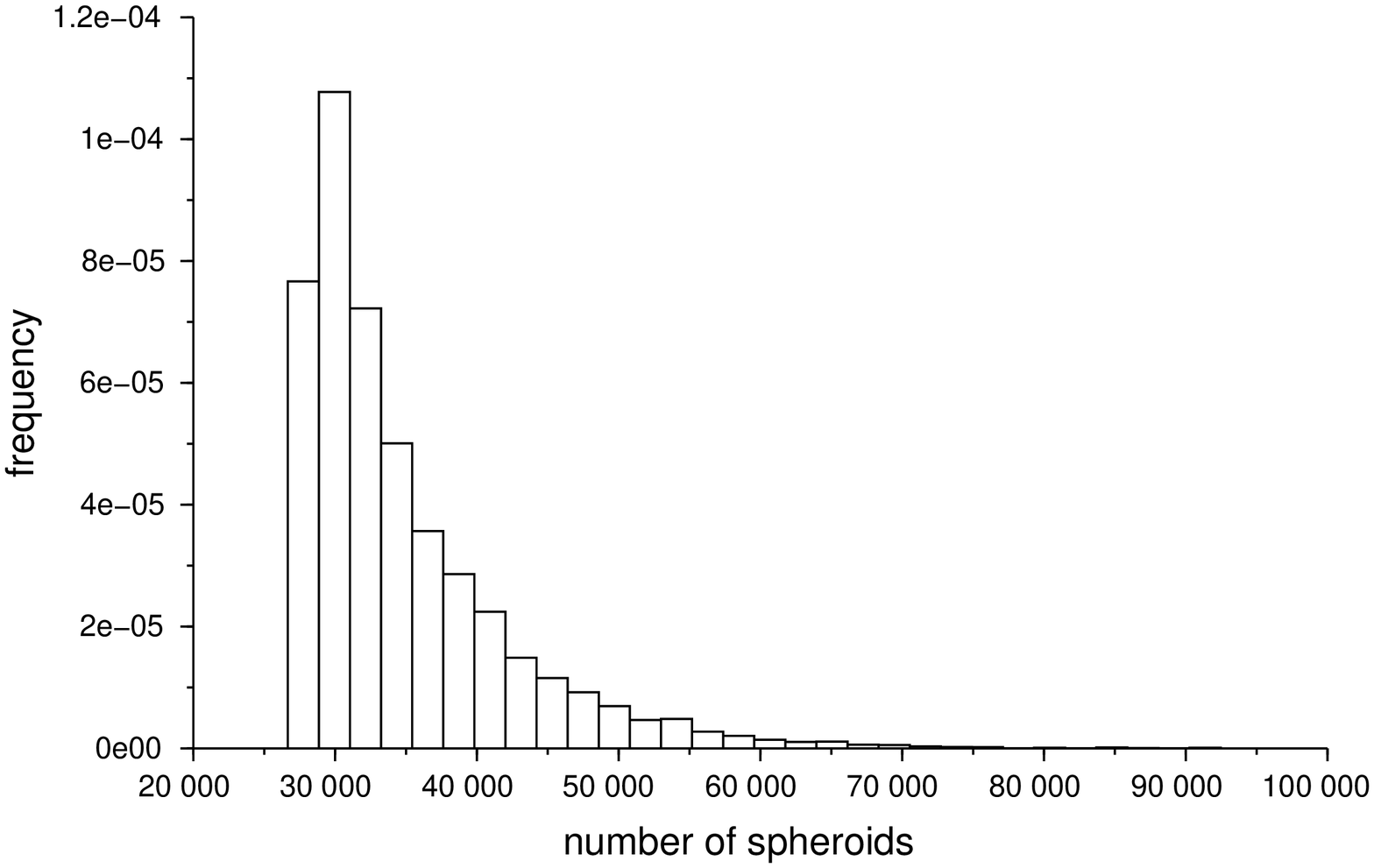}\includegraphics[width=8cm]{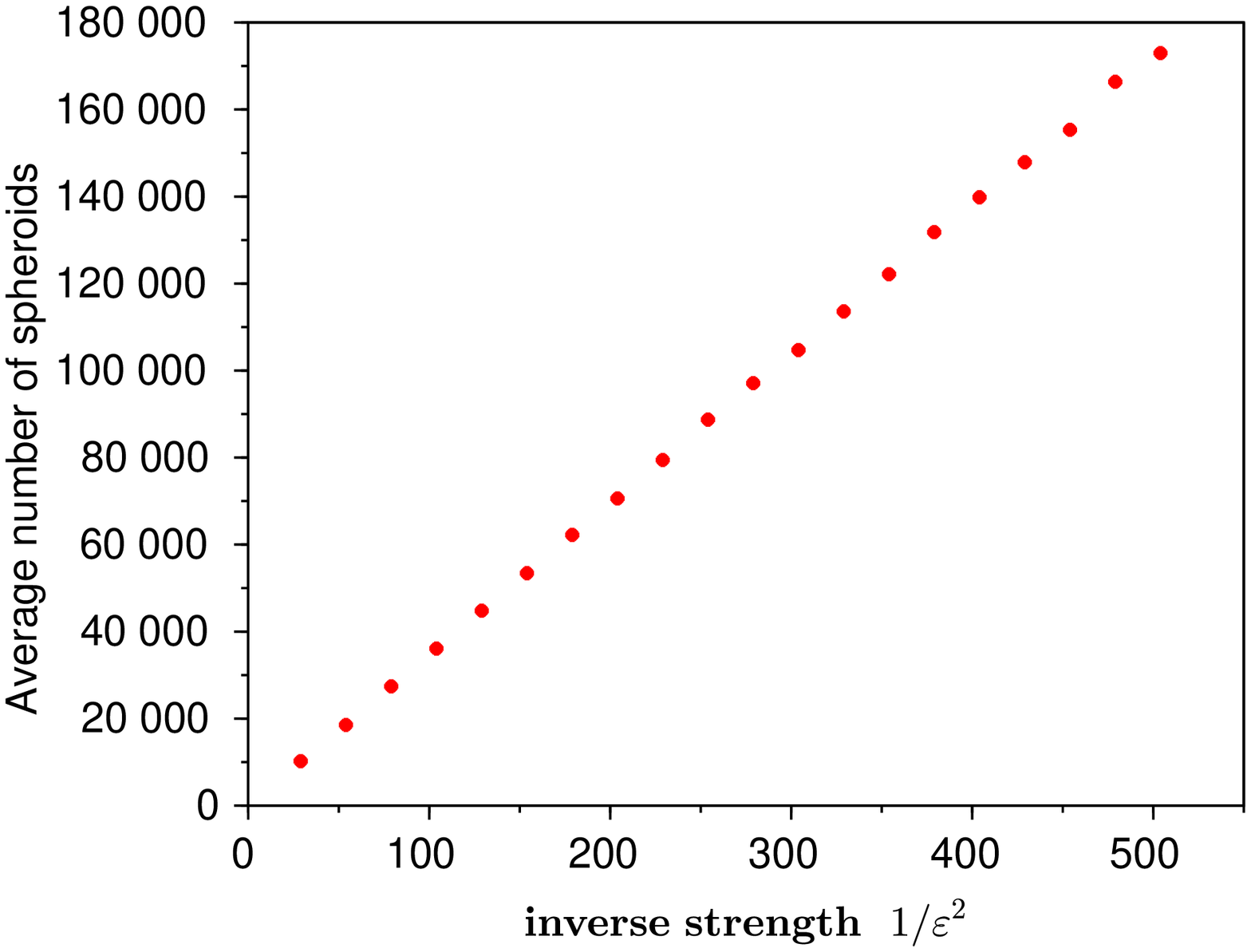}}
\caption{\small Histogram of the number of $\phi_\varepsilon$-domains used to cover the time interval $[0,1]$ for $\varepsilon=0.1$ (left, sample of size $10\,000$) -- Average number of $\phi_\varepsilon$-domains versus the inverse strength $1/\varepsilon^2$ (right, sample of size for each point: $1\,000$). For both pictures:  $x_0=0$.}
\label{Fig}
\end{figure}

As said before, it is quite difficult to compare such a method with other numerical approximations of diffusion processes: the other methods don't lead to build paths which are a.s. $\varepsilon$-close to the diffusion ones. Let us nevertheless sketch a rough comparison: the simulation of $10\,000$ paths on $[0,1]$ with $\varepsilon=0.1$ requires about $255.7$ sec %CLOCKS_PER_SEC 1000000 comptime 2.55715e+08
and one can observe that the average time step is about $3\cdot 10^{-5}$.      %2,98 
If we consider the classical Euler-scheme with the corresponding constant step size, then a similar sample of paths requires about $41.3$ sec %CLOCKS_PER_SEC 1000000 comptime 4.13481e+07
(on the same computer).  One argument which permits to explain the difference in speed is that the $\varepsilon$-strong approximation needs at each step to test if the number of $\phi_\varepsilon$-domains used so far is sufficient to cover the time interval, such a test is quite time-consuming. Let us also note that the $\varepsilon$-strong approximation permits to be precise not only in the approximation of the marginal distribution but also in the approximation of the whole trajectory. In other words, it is a useful tool for Monte Carlo estimation of an integral, of a supremum, of any functional of the diffusion.

This first example illustrates also the convergence result presented in Theorem \ref{thm:efficient:diff}. Since the limiting value is expressed as an average integral of a Brownian motion path,  the use of the Monte Carlo procedure permits to get an approximated value: 347.1 on one hand and on the other hand the estimation of the regression line in Figure \ref{Fig} (right) indicates  
\[
\mathcal{M}(\hat{N}^\varepsilon_1)\approx 344.3\times \frac{1}{\varepsilon^2}+418.5
\]
%a  = 344.27336  b  = 418.46674
where $\mathcal{M}$ corresponds to the estimated average value for the sample of size $10\,000$. \\[5pt]

\noindent {\bf Example 2 (polynomial decrease).} We consider on the time interval $[0,1]$ the $\varepsilon$-strong approximation of the mean reverting diffusion process given by
\begin{equation}
\label{def:ex:2}
a(t)=\frac{1}{2}\frac{1}{1+t},\quad b(t)=0,\quad\mbox{and}\quad  \overline{\sigma}(t)=2.
\end{equation}
Then we obtain the time-scale function $\rho(t)=4\ln(1+t)$ and
\(
f(t,x)=x\sqrt{1+t}.
\)
We choose therefore $F(x)=\sqrt{2}+\frac{\sqrt{|x|}}{2}$. The number of $\phi_\varepsilon$-domains is illustrated in Fig \ref{Fig2}. The simulation of a sample of trajectories on the time interval $[0,1]$ of size $10\,000$ requires also about 261 sec for the particular choice $\varepsilon=0.01$ while the classical Euler scheme generated with a comparable step size $3.3\cdot 10^{-5}$ %3,31 
requires $30$ sec. %30.46 

\begin{figure}[h]
\centerline{\includegraphics[width=8cm]{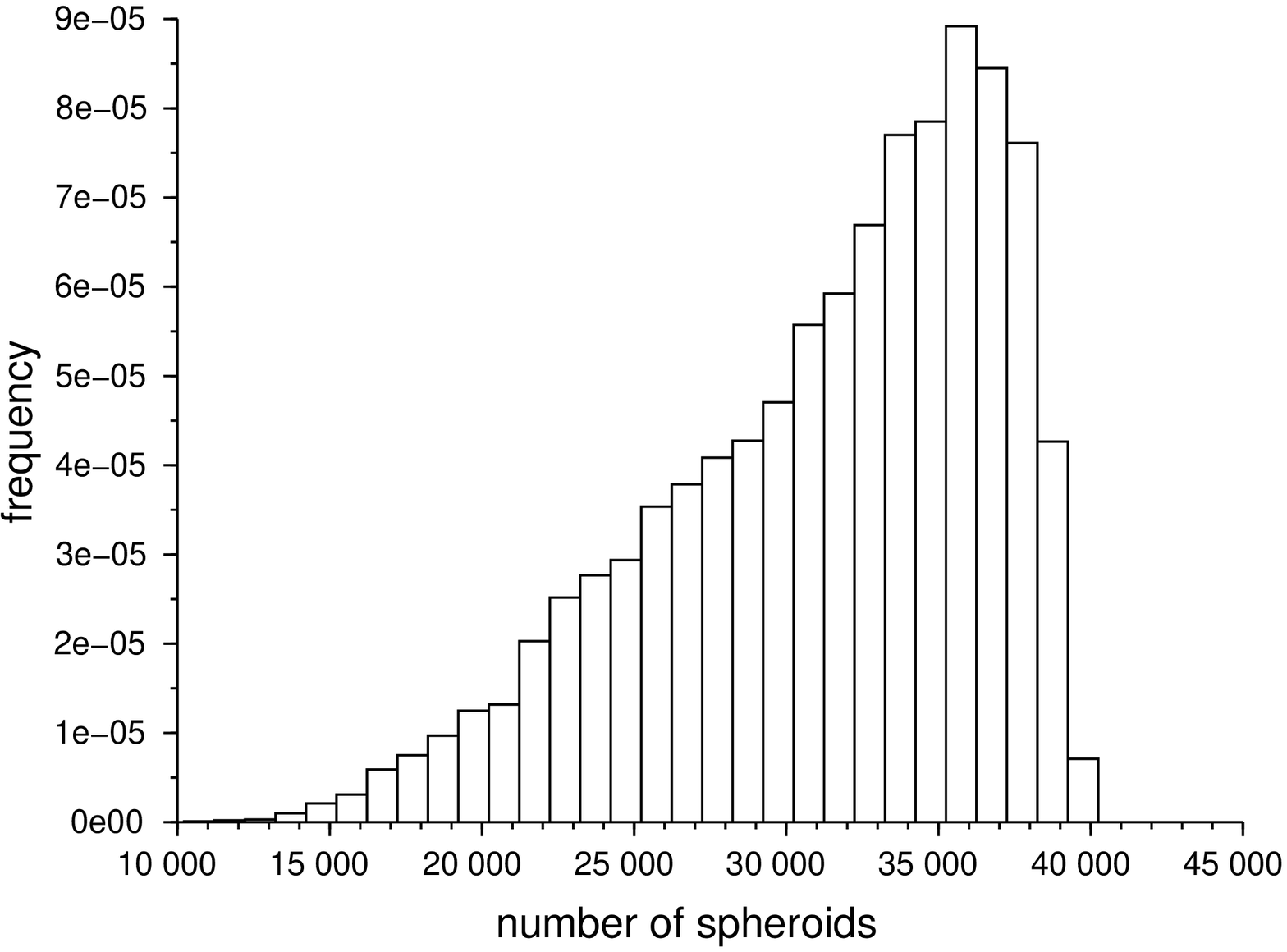}\includegraphics[width=8cm]{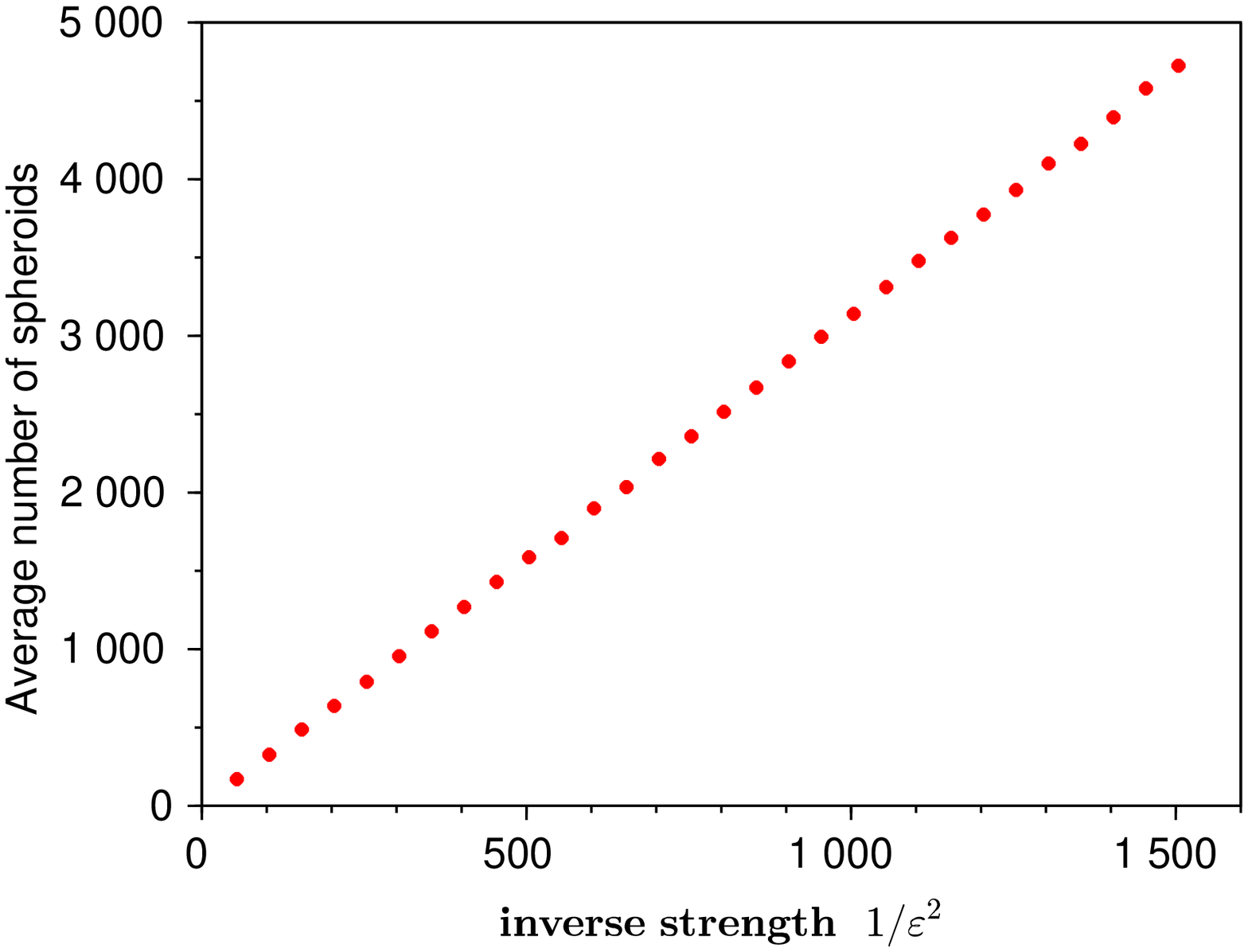}}
\caption{\small Histogram of the number of $\phi_\varepsilon$-domains used to cover the time interval $[0,1]$ for $\varepsilon=0.01$ (left, sample size: $10\,000$) -- Average number of $\phi_\varepsilon$-domains versus the inverse strength $1/\varepsilon^2$ (right, sample size $1\,000$). For both pictures:  $x_0=0$.}
\label{Fig2}
\end{figure}

%\section{Approximation of Bessel paths}
%Ici il faudrait d\'ecrire l'approximation \`a  la fois pour le Bessel de dimension enti\`ere et pour le Bessel de dimension non enti\`ere. Pour ce dernier, il est utile d'explorer un algo diff\'erent: ajuster les temps moyens de sortie pour chacun des deux processus intervenant dans la d\'ecomposition au lieu d'ajuster les supports des lois de probabilit\'e. A voir si c'est possible...
%\section{Applications and examples}
%\subsection{Monte-Carlo estimation for Asian options}
%\subsection{Approximation of hitting time densities}
\section*{Appendix}
Let us just present here useful upper-bounds related to the log-gamma distribution. % 
\begin{lemma}\label{lem1} Let $\alpha \geq 1 $ and $\beta \geq 1$ and let us assume that $W$ is a random variable of a log-gamma distribution type. Its 
%%% (see Asmunsen book, p. 10)
 probability distribution function is
\[
f_W(t):=\frac{1}{\Gamma(\alpha)\beta^\alpha}(-\ln t)^{\alpha-1}t^{1/\beta-1}1_{[0,1]}(t),\quad \forall t\in\mathbb{R}.
\]
\begin{itemize}
\item[(1)]  Then
\begin{equation}\label{eq:laplace}
\mathbb{E}[e^{-\lambda W}]=\sum_{k\ge 0}\frac{(-1)^k\lambda^k}{k!(1+k\beta)^\alpha}.
\end{equation}
\item[(2)]  In particular, for $\alpha=1$, we get
\begin{equation}\label{eq:bound=alpha1}
\mathbb{E}[e^{-\lambda W}]\le \frac{\Gamma(1/\beta)}{\beta \lambda^{1/\beta}}.
\end{equation}
\item[(3)]  In the case: $\alpha>1$, for any $\beta'>\beta$, we obtain
\begin{equation}\label{eq:bound>alpha}
\mathbb{E}[e^{-\lambda W}]\le \omega(\alpha,\beta,\beta')\frac{1}{\lambda^{1/\beta'}},\quad\mbox{with}\  \omega(\alpha,\beta,\beta') % 
=\frac{(\alpha-1)^{\alpha-1}\Gamma(1/\beta')}{\Gamma(\alpha)\beta^\alpha e^{\alpha-1}(\beta^{-1}-\beta'^{-1})^{\alpha-1}}.
\end{equation}
\end{itemize}
\end{lemma}
\begin{proof}
For $(1)$, let us first note that an easy computation leads to the following moments, for any $k\ge 1$:
\begin{equation}
\label{moment-log-gamma}
\mathbb{E}[W^k]=(1+k\beta)^{-\alpha}.
\end{equation}
After summing over $k$ \eqref{moment-log-gamma} we deduce the expression of the Laplace transform \eqref{eq:laplace}. \\
For $(2)$, let us first consider the particular case: $\alpha=1$. Using the expression of the PDF and the change of variable $u=\lambda x$, we obtain
\begin{align*}
\mathbb{E}[e^{-\lambda W}]&=\frac{1}{\beta}\int_0^1 e^{-\lambda x} x^{1/\beta-1}\,\dint x=\frac{\Gamma(1/\beta)}{\beta \lambda^{1/\beta}}\,I_\beta(\lambda),
\end{align*}
where $I_\beta(\lambda):=\frac{1}{\Gamma(1/\beta)}\int_0^\lambda e^{-u}u^{1/\beta-1} \dint u$. We observe that $\lambda\mapsto I_\beta(\lambda)$ is increasing and $\lim_{\lambda\to\infty}I_\beta(\lambda)=1$ which leads to \eqref{eq:bound=alpha1}. \\
For $(3)$, let us now assume that $\alpha>1$ and consider $\beta'>\beta$. Then
\begin{align*}
\mathbb{E}[e^{-\lambda W}]&=\frac{1}{\Gamma(\alpha)\beta^\alpha}\int_0^1 e^{-\lambda x} \Big( -x^{(1/\beta-1/\beta')/(\alpha-1)}\ln x \Big)^{\alpha-1}x^{1/\beta'-1}\,\dint x\\
&\le \left(\frac{\alpha-1}{e\beta^{\alpha/(\alpha-1)}(\beta^{-1}-\beta'^{-1})}\right)^{\alpha-1}\frac{1}{\Gamma(\alpha)}\int_0^1 e^{-\lambda x}x^{1/\beta'-1}\dint x,
\end{align*}
since $-u^r\ln u\le (re)^{-1}$. Moreover the change of variable $u=\lambda x$ leads to
\[
\mathbb{E}[e^{-\lambda W}]\le \frac{\omega(\alpha,\beta,\beta')}{\lambda^{1/\beta'}}\, I_{\beta'}(\lambda).
\]
The bound $I_{\beta'}(\lambda)\le 1$ directly leads to \eqref{eq:bound>alpha}. 

\end{proof}

\bibliographystyle{plain}

\end{document}